\documentclass[english]{elsarticle}
\usepackage[T1]{fontenc}
\usepackage[latin9]{inputenc}
\usepackage{geometry}
\usepackage{mathtools}
\usepackage{amsmath}
\usepackage{amsthm}
\usepackage{amssymb}
\usepackage{stackrel}
\usepackage{graphicx}

\makeatletter

\theoremstyle{plain}
\ifx\thesection\undefined
\newtheorem{thm}{\protect\theoremname}
\else
\newtheorem{thm}{\protect\theoremname}[section]
\fi
\theoremstyle{remark}
\ifx\thesection\undefined
\newtheorem{rem}{\protect\remarkname}
\else
\newtheorem{rem}{\protect\remarkname}[section]
\fi
\theoremstyle{plain}
\ifx\thesection\undefined
\newtheorem{lem}{\protect\lemmaname}
\else
\newtheorem{lem}{\protect\lemmaname}[section]
\fi

\allowdisplaybreaks
\usepackage{xcolor}
\usepackage[colorlinks,
linkcolor=blue,
anchorcolor=yellow,
citecolor=red,
]{hyperref}

\@ifundefined{showcaptionsetup}{}{%
 \PassOptionsToPackage{caption=false}{subfig}}
\usepackage{subfig}
\makeatother

\usepackage{babel}
\providecommand{\lemmaname}{Lemma}
\providecommand{\remarkname}{Remark}
\providecommand{\theoremname}{Theorem}

\begin{document}

\begin{frontmatter}{}

\title{Stability and error analysis of the SAV schemes for the inductionless MHD equations}

\author[zzu,zzu1]{Xiaodi Zhang}

\ead{zhangxiaodi@lsec.cc.ac.cn}

\address[zzu]{Henan Academy of Big Data, Zhengzhou University, Zhengzhou 450052,
China.}

\address[zzu1]{School of Mathematics and Statistics, Zhengzhou University, Zhengzhou
450001, China.}

\author[xju]{Xianghai Zhou}
\ead{zxhmath166@163.com}

\address[xju]{College of Mathematics and System Sciences,
	Xinjiang University, Urumqi 830046, China.}

\begin{abstract}
In this paper, we consider numerical approximations for solving the
inductionless magnetohydrodynamic (MHD) equations. By utilizing the
scalar auxiliary variable (SAV) approach for dealing with the convective and coupling terms, we propose some first- and second-order
schemes for this system. These schemes are linear, decoupled, unconditionally
energy stable, and only require solving a sequence of differential
equations with constant coefficients at each time step. We further
derive a rigorous error analysis for the first-order scheme, establishing
optimal convergence rates for the velocity, pressure, current density
and electric potential in the two-dimensional case. Numerical examples
are presented to verify the theoretical findings and show the performances
of the schemes.
\end{abstract}
\begin{keyword}
inductionless MHD equations, SAV, energy stable, decoupled,
error estimates 
\end{keyword}

\end{frontmatter}{}

\section{Introduction}

The incompressible MHD describes the dynamic behavior
of an electrically conducting fluid under the influence of a magnetic
field. It has been widely used in many science and engineering applications,
such as liquid metal cooling for nuclear reactors, and sustained plasma
confinement for controlled thermonuclear fusion, see \citep{Abdou01,Dav01,Gerbeau06}.
Mathematically, the most frequently used model is obtained by coupling
the Navier-Stokes equations for hydrodynamics with the Maxwell equations
for electromagnetism. However, in most terrestrial applications, the
magnetic Reynolds number of MHD flows is small. Consequently, the
magnetic induction can usually be negligible compared with the external
magnetic field and the electric field is considered to be quasi-static.
The MHD equations with this simplification are referred as the inductionless
MHD equations. Much effort has been spent on theoretical analysis
and mathematical modeling of the inductionless MHD equations, see
\citep{Badia14,Zhang2021} and the references therein. 

Let $\Omega\subset\mathbb{R}^{d},\,d=2,3$ be a bounded domain with
Lipschitz-continuous boundary $\Gamma\coloneqq\partial\Omega$. In
this paper, we consider the incompressible inductionless MHD equations
as follows,
\begin{subequations}
\begin{align}
\boldsymbol{u}_{t}-R_{e}^{-1}\Delta\boldsymbol{u}+\boldsymbol{u}\cdot\nabla\boldsymbol{u}+\nabla p-\kappa\boldsymbol{J}\times\boldsymbol{B} & =\boldsymbol{0}\quad\text{ in }\Omega\times(0,T],\label{eq:modelu}\\
\mathrm{div}\boldsymbol{u} & =0\quad\text{ in }\Omega\times(0,T],\label{eq:modeldivu}\\
\boldsymbol{J}+\nabla\phi-\boldsymbol{u}\times\boldsymbol{B} & =\boldsymbol{0}\quad\text{ in }\Omega\times(0,T],\label{eq:modelJ}\\
\mathrm{div}\boldsymbol{J} & =0\quad\text{ in }\Omega\times(0,T],\label{eq:modeldivJ}
\end{align}
\label{eq:modelmhd}
\end{subequations}

\noindent where $T>0$ is the terminal time, \textbf{$\boldsymbol{u}$
}is the fluid velocity, $p$ the hydrodynamic pressure, $\boldsymbol{J}$
the current density, $\phi$ the electric potential, $R_{e}$ is the
fluid Reynolds number, $\kappa$ is the coupling number. Function
$\boldsymbol{B}$ is the applied magnetic field which is assumed to
be given. The system is considered in conjunction with the following
initial and boundary conditions,
\begin{align}
\boldsymbol{u}(x,0)=\boldsymbol{u}^{0}(x) & \text{ in }\Omega,\label{eq:init}\\
\boldsymbol{u}=\boldsymbol{0},\quad\boldsymbol{J}\cdot\boldsymbol{n}=0 & \text{ on }\Gamma\times(0,T],\label{eq:bdy}
\end{align}
where the initial value satisfies $\nabla\cdot\boldsymbol{u}_{0}=0$
and $\boldsymbol{n}$ is the unit outer normal vector on $\Gamma$. 

Numerical solving the inductionless MHD model has drawn a considerable
amount of attention. Since it replaces Maxwell\textquoteright s equations
with Poisson\textquoteright s equation for the electric potential,
numerical solution is more economic compared with the full MHD model.
In particular, the inductionless MHD model is further simplified as
the reduced MHD model by eliminating the current density $\boldsymbol{J}$.
In \citep{Yuk12,Yuk15}, Yuksel et al. studied the error analysis
for both semi-discrete and fully discrete finite element approximate
of the reduced MHD model. In \citep{Ni071,Ni072,Ni12}, Ni and his
collaborators studied a class of consistent and charge-conservative
finite volume scheme for the inductionless MHD equations on both structured
and unstructured meshes. In \citep{Pla11}, Planas et al. proposed
a stabilized finite element method to solve the inductionless MHD
problem, in which the time discretization is based on the Backward-Euler
method. In 2019, Li et al. \citep{Li19} proposed a fully discrete
and charge-conservative finite element method and provided a plain
convergence analysis. Later, Long further presented optimal error
estimates for both the semi-discrete and full-discrete scheme \citep{Long19}.
Most of the schemes mentioned above are coupled-type where they need to assemble
and solve a multi-physics and large system at each time step. Thus,
it could be computationally expensive in numerical computation, especially
in three dimensions. 

To address this issue, some decoupled methods have been investigated
attractive in the literature. Since decoupled methods usually solve
the coupled problem by successively solving the sub-physics problems
at each step and many efficient solvers can be used for each of them.
However, decoupled methods are more computationally economical but may
lose some stability. Thus, it is desirable to design decoupled methods
while preserving the energy stability, in the sense that the discrete
energy dissipation laws hold. For the reduced MHD model, Layton et
al. \citep{Layton2013} introduced two partitioned methods and studied
the stability, where the first order method is shown to be unconditionally
stable and the second order method is shown to be conditionally stable.
In 2021, Zhang et al. \citep{Zhang2020} proposed and analyzed a decoupled,
unconditionally energy stable and charge-conservative finite element
method. By adding a first-order stabilized term to the magnetic problem
and making some subtle implicit-explicit treatments for coupling terms,
their scheme can decouple the computation of the magnetic problem
from the fluid problem while ensuring the energy stability. In these
works, the nonlinear terms are treated either implicitly or semi-implicitly,
so that one needs to solve a nonlinear system or some linear systems with
variable coefficients at each time step. Thus, it is desirable to
be able to treat the nonlinear term explicitly while maintaining energy
stability. With such treatment, the schemes only require the solution
of linear systems with constant coefficients upon discretization,
and thus are very efficient and popular for dynamical simulations. 

Recently, SAV based schemes have gained much attention recently due
to their efficiency, flexibility and accuracy. The SAV approach was
first studied in \citep{Shen2018,Shen2018a} to construct efficient
schemes for gradient flows. Nowadays, it has been a powerful approach to develop energy stable numerical schemes for general dissipative systems,
such as Navier-Stokes equations \citep{Lin2019,Li2020b}, magnetohydrodynamic
equations \citep{Li2021,Yang2021m} and Cahn-Hilliard-Navier-Stokes
equations \citep{Yang2021c,Li2020a}. Based on an auxiliary variable
associated with the total system energy, $q(t)=\sqrt{{\rm E}_{\text{NS}}(t)}$
with ${\rm E}_{\text{NS}}(t)=\frac{1}{2}\left\Vert \boldsymbol{u}\right\Vert ^{2}$,
Dong et al. \citep{Lin2019} constructed a numerical scheme for the
NS equations. Within each time step, the scheme involves the computations
of two generalized Stokes equations with constant coefficient matrices,
together with a nonlinear algebraic equation a for the auxiliary variable.
To address the theoretical and practical issues form the nonlinear
algebraic equation, Li et al. \citep{Li2020b} proposed and analyzed
some first- and second-order pressure correction schemes using the
SAV approach for the NS equations, where $q(t)=\exp(-t/T)$. Later,
they extend the proposed approach and theoretical findings to the
MHD equations in \citep{Li2021} and the Cahn-Hilliard-Navier-Stokes
system in \citep{Li2020a} for dealing with the nonlinear and coupling
terms that satisfy the \textquotedblleft zero-energy-contribution\textquotedblright{}
feature. Meanwhile, Yang \citep{Yang2021c,Yang2021e} designed a series
of linear and energy stable schemes for the flow-coupled phase-field
models. In these works, the auxiliary variable for the NS equations
is simpler, $q(t)=1$. Following the same idea, this approach was
later extended to devise fully decoupled finite element schemes
for the MHD equations in \citep{Yang2021m,Zhang2022a}. 

The purpose of this paper is to propose and analyze some SAV schemes for the inductionless MHD equations. By utilizing the SAV approach
for dealing with the convective term and coupling terms, some first-
and second-order schemes are constructed for this system. These schemes
are linear, decoupled, unconditionally energy stable, and only require
solving a sequence of differential equations with constant coefficients
at each time step. Thus, they are very efficient and easy to implement.
We further establish rigorous unconditional energy stability and error
analysis for the first-order scheme in the two-dimensional case. Some
numerical experiments are provided to confirm the predictions of the
theory and demonstrate the efficiency of the proposed schemes.

While the construction of the SAV schemes for the inductionless MHD
equations is quite straightforward, it is much more difficult to carry
out the error analysis as we have to deal with the issues due to the
non-local coupling between the SAV and other variables, and the explicit
treatment of the convective terms and coupling terms. It is also remarkable
that while the error analysis is somewhat similar to the ones in \citep{Li2020b,Li2021},
the extension of error analysis for the inductionless MHD equations
is still non-trivial. On the one hand, compared to the Navier-Stokes
equations, the error analysis for the inductionless MHD equations
is much more involved due to the coupling terms. On the other hand,
compared to the full MHD equations, the inductionless MHD equations
is a hybrid system. More precisely, the
fluid problem is unsteady while the electromagnetic problem is steady.
The lack of the derivative term to time in electromagnetic
problem makes the error analysis more complicated and tough. Therefore,
more delicate analyses are needed for the error estimates. 

The paper is organized as follows. In Section \ref{sec:Pre}, we introduce
some notations and present the energy estimate for the inductionless
MHD equations. In Section \ref{sec:Schemes}, we propose the SAV schemes
and prove the unconditional stability. In Section \ref{sec:Error},
we carry out a rigorous error analysis for the first-order scheme
in the two-dimensional case. In Section \ref{sec:Num}, we present
some numerical experiments. In Section \ref{sec:Conclud}, we conclude
with a few remarks.

\section{Preliminaries\label{sec:Pre}}

We begin with introducing some notations and Sobolev spaces. For all $1\le q\le \infty$, $L^{q}(\Omega)$ for the $q$-integrable function space with the norm $\left\Vert\cdot\right\Vert_{0,q}$. Particularly, $L^{2}(\Omega)$ is equipped with the inner product
$(\cdot,\cdot)$ and norm $\left\Vert \cdot\right\Vert $. The subspace
of $L^{2}(\Omega)$ with zero mean value over $\Omega$ is further denoted
as $L_{0}^{2}(\Omega)$. For all $m\in\mathbb{N}^{+},1\le q\le\infty$,
let $W^{m,q}(\Omega)$ denote the standard Sobolev space equipped
with the standard Sobolev norm $\left\Vert \cdot\right\Vert _{m,q}$.
For $q=2$, we write $H^{m}(\Omega)$ for $W^{m,2}(\Omega)$ and its
corresponding norm is $\left\Vert \cdot\right\Vert _{m}$. Let $\boldsymbol{H}({\rm div},\Omega)$
be the subspace of $\boldsymbol{L}^{2}(\Omega)$ with square integrable
divergence, the norm is defined by $\left\Vert \cdot\right\Vert _{{\rm div}}$.
Spaces $H_{0}^{1}(\Omega)$ and $\boldsymbol{H}_{0}(\mathrm{div},\Omega)$
denote their sub-spaces with vanishing traces and vanishing normal
traces on $\Gamma$, respectively. For a given Sobolev space $X$,
we write $L^{q}(0,T;X)$ for the Bochner space and its norm is written
by $\left\Vert \cdot\right\Vert _{L^{q}(0,T;X)}$. Here and what follows,
we use $C$ to denote generic positive constants independent of the
discretization parameters, which may take different values at different places.

For convenience, we introduce some notations for function spaces 
\[
\boldsymbol{X}:=\boldsymbol{H}_{0}^{1}(\Omega),\quad\boldsymbol{V}:=\{\boldsymbol{v}\in\boldsymbol{X}:\nabla\cdot\boldsymbol{v}=0\},\quad Y=L_{0}^{2}(\Omega),\quad\boldsymbol{D}:=\boldsymbol{H}_{0}({\rm div},\Omega),\quad S=L_{0}^{2}(\Omega).
\]
We will use the following trilinear form, 
\[
b(\boldsymbol{u},\boldsymbol{v},\boldsymbol{w})=\left(\boldsymbol{u}\cdot\nabla\boldsymbol{v},\boldsymbol{w}\right)\quad\forall\boldsymbol{u},\boldsymbol{v},\boldsymbol{w}\in\boldsymbol{X}.
\]
It is easy to see that the trilinear form $b(\cdot,\cdot,\cdot)$
is a skew-symmetric with respect to its last two arguments,
\begin{equation}
b(\boldsymbol{u},\boldsymbol{v},\boldsymbol{w})=-b(\boldsymbol{u},\boldsymbol{w},\boldsymbol{v})\quad\forall\boldsymbol{u}\in\boldsymbol{V},\quad\boldsymbol{v},\boldsymbol{w}\in\boldsymbol{X},\label{eq:e_skew-symmetric1}
\end{equation}
and 
\begin{equation}
b(\boldsymbol{u},\boldsymbol{v},\boldsymbol{v})=0\quad\forall\boldsymbol{u}\in\boldsymbol{V},\quad\boldsymbol{v}\in\boldsymbol{X}.\label{eq:e_skew-symmetric2}
\end{equation}

Now, we are in a position to establish the energy estimate for the
inductionless MHD system. By taking the $\boldsymbol{L}^{2}$-inner
product of $\boldsymbol{u}$ with (\ref{eq:modelu}) and using the
integration by parts and (\ref{eq:modeldivu}), we get
\begin{equation}
\frac{1}{2}\frac{d}{dt}\left\Vert \boldsymbol{u}\right\Vert ^{2}+R_{e}^{-1}\left\Vert \nabla\boldsymbol{u}\right\Vert ^{2}+\kappa(\boldsymbol{J}\times\boldsymbol{B},\boldsymbol{u})=0.\label{eq:energyu}
\end{equation}
By taking the $\boldsymbol{L}^{2}$-inner product of $\kappa\boldsymbol{J}$
with (\ref{eq:modelJ}) and using the integration by parts and (\ref{eq:modeldivJ}),
we have
\begin{equation}
\kappa\left\Vert \boldsymbol{J}\right\Vert ^{2}-\kappa(\boldsymbol{u}\times\boldsymbol{B},\boldsymbol{J})=0.\label{eq:energyJ}
\end{equation}
By combining (\ref{eq:energyu})-(\ref{eq:energyJ}) and using (\ref{eq:e_skew-symmetric2}),
we obtain the law of energy dissipation that reads as,
\begin{equation}
\frac{d}{dt}\mathrm{E}(t)=-R_{e}^{-1}\left\Vert \nabla\boldsymbol{u}\right\Vert ^{2}-\kappa\left\Vert \boldsymbol{J}\right\Vert ^{2}\quad\text{with}\quad\mathrm{E}(t)=\frac{1}{2}\left\Vert \boldsymbol{u}\right\Vert ^{2}.\label{eq:energy}
\end{equation}
The energy law describes the variation of the total energy caused
by energy conversion. Since the inducted magnetic field is neglected
and the electric field is considered to be quasi-static, the total
energy $\mathrm{E}$ only consists of the fluid kinetic energy $\frac{1}{2}\left\Vert \boldsymbol{u}\right\Vert ^{2}$.
The dissipation of $\mathrm{E}$ stems from the friction losses $R_{e}^{-1}\left\Vert \nabla\boldsymbol{u}\right\Vert ^{2}$
and the Ohmic losses $\kappa\left\Vert \boldsymbol{J}\right\Vert ^{2}$.
The above proof to obtain the law of energy dissipation (\ref{eq:energy})
lies on the following two identities,
\begin{equation}
\left(\boldsymbol{u}\cdot\nabla\boldsymbol{u},\boldsymbol{u}\right)=0,\quad\kappa\left(\boldsymbol{J}\times\boldsymbol{B},\boldsymbol{u}\right)-\kappa\left(\boldsymbol{u}\times\boldsymbol{B},\boldsymbol{J}\right)=0.\label{eq:zero}
\end{equation}
These two equities can be regarded as the contribution of two types
of coupling terms to the total free energy of the system is zero.
These unique \textquotedblleft zero-energy-contribution\textquotedblright{}
property will be used to design decoupling type numerical schemes.

\section{The SAV schemes\label{sec:Schemes} }

In this section, we first reformulate the inductionless MHD model
into an equivalent system with SAV. Then, we construct first- and
second-order SAV schemes and prove that they are unconditionally energy
stable.

Let $\left\{ t^{n}=n\tau:\,n=0,1,\cdots,N\right\} $, $\tau=T/N$,
be an equidistant partition of the time interval $[0,T].$ We denote
$(\cdot)^{n}$ as the variable $(\cdot)$ at time step $n.$ For any
function $v(x,t)$, define 
\[
\delta_{t}v^{n+1}=\frac{v^{n+1}-v^{n}}{\tau},\quad\delta_{t}^{2}v^{n+1}=\frac{3v^{n+1}-4v^{n}+v^{n-1}}{2\tau},\quad\hat{v}^{n+1}=2v^{n}-v^{n-1}
\]
In particular, when $n=0$, we denote 
\begin{equation}
\delta_{t}^{2}v^{1}=\frac{v^{1}-v^{0}}{\tau},\quad\hat{v}^{1}=v^{0}.\label{eq:BDF0}
\end{equation}

\subsection{Reformulated system\label{sec:Reform}}

Inspired by \citep{Li2020b}, we introduce a scalar auxiliary variable
$q(t)$,
\begin{equation}
q(t)=\exp\left(-\frac{t}{T}\right).\label{eq:SAV}
\end{equation}
Note that $q(t)$ is a scalar-valued number, not a field function.
This function will serve as the scalar auxiliary variable. By
taking the derivative of (\ref{eq:SAV}) with respect to $t$, we
obtain
\[
\frac{\mathrm{d}q}{\mathrm{d}t}=-\frac{q}{T}.
\]
Observed that this a linear and dissipative ordinary differential
equation for the scalar auxiliary variable. This feature is vital
for designing unconditionally energy-stable and linear schemes. 

In light of equation $q\exp\left(\frac{t}{T}\right)=1$ and (\ref{eq:zero}),
we rewrite the system (\ref{eq:modelmhd}) into the following form
\begin{subequations}
\begin{align}
\boldsymbol{u}_{t}-R_{e}^{-1}\Delta\boldsymbol{u}+\nabla p+q\exp\left(\frac{t}{T}\right)\left(\boldsymbol{u}\cdot\nabla\boldsymbol{u}-\kappa\boldsymbol{J}\times\boldsymbol{B}\right) & =\boldsymbol{0},\label{modeleq:u}\\
\mathrm{div}\boldsymbol{u} & =0,\label{modeleq:divu}\\
\boldsymbol{J}+\nabla\phi-q\exp\left(\frac{t}{T}\right)\boldsymbol{u}\times\boldsymbol{B} & =\boldsymbol{0},\label{modeleq:J}\\
\mathrm{div}\boldsymbol{J} & =0,\label{modeeq:divJ}\\
\frac{\mathrm{d}q}{\mathrm{dt}}+\frac{q}{T}-\exp\left(\frac{t}{T}\right)\left(\left(\boldsymbol{u}\cdot\nabla\boldsymbol{u},\boldsymbol{u}\right)-\kappa\left(\boldsymbol{u}\times\boldsymbol{B},\boldsymbol{J}\right)-\kappa\left(\boldsymbol{J}\times\boldsymbol{B},\boldsymbol{u}\right)\right) & =0,\label{mdoeleq:q}
\end{align}
\label{eq:SAVmodel}
\end{subequations}
Note that in the reformulated system, $q(t)$ is treated as an approximation
of $\exp\left(-\frac{t}{T}\right)$ and is computed by solving this
system of equations, not by using equation (\ref{eq:SAV}). The initial
condition for $q(t)$ is set as $q(0)=1$. The last term in the equation
for $q$ of (\ref{mdoeleq:q}) is added to balance the nonlinear term
and coupling term in (\ref{eq:SAVmodel}) in the discretized case.
We next focus on this reformulated system, and present unconditionally
energy-stable schemes for this system.

\medskip{}

\begin{thm}
The reformulated system (\ref{eq:SAVmodel}) admits the following
law of energy dissipation,
\begin{equation}
\frac{d}{dt}\mathrm{E}_{\text{SAV}}(t)=-R_{e}^{-1}\left\Vert \nabla\boldsymbol{u}\right\Vert ^{2}-\kappa\left\Vert \boldsymbol{J}\right\Vert ^{2}-\frac{1}{T}\left|q\right|^{2}\quad\text{with}\quad\mathrm{E}_{\text{SAV}}(t)=\frac{1}{2}\left\Vert \boldsymbol{u}\right\Vert ^{2}+\frac{1}{2}\left|q\right|^{2}.\label{eq:SAVenergy}
\end{equation}
\end{thm}
\begin{proof}
Taking the $\boldsymbol{L}^{2}$-inner product of $\boldsymbol{u}$
with (\ref{modeleq:u}), using the integration by parts and (\ref{modeleq:divu}),
we obtain 
\begin{equation}
\frac{1}{2}\frac{d}{dt}\left\Vert \boldsymbol{u}\right\Vert ^{2}+R_{e}^{-1}\left\Vert \nabla\boldsymbol{u}\right\Vert ^{2}+q\exp\left(\frac{t}{T}\right)\left(\boldsymbol{u}\cdot\nabla\boldsymbol{u},\boldsymbol{u}\right)+q(t)\exp\left(\frac{t}{T}\right)\kappa(\boldsymbol{J}\times\boldsymbol{B},\boldsymbol{u})=0.\label{eq:SAVenergyu}
\end{equation}
Taking the $\boldsymbol{L}^{2}$-inner product of $\kappa\boldsymbol{J}$
with (\ref{modeleq:J}), using the integration by parts and (\ref{modeeq:divJ}),
we obtain 
\begin{equation}
\kappa\left\Vert \boldsymbol{J}\right\Vert ^{2}-q\exp\left(\frac{t}{T}\right)\kappa(\boldsymbol{u}\times\boldsymbol{B},\boldsymbol{J})=0.\label{eq:SAVenergyJ}
\end{equation}
Multiplying $q$ with (\ref{mdoeleq:q}) leads to
\begin{equation}
\frac{1}{2}\frac{\mathrm{d}}{\mathrm{dt}}\left|q\right|^{2}+\frac{1}{T}\left|q\right|^{2}-q\exp\left(\frac{t}{T}\right)\left(\left(\boldsymbol{u}\cdot\nabla\boldsymbol{u},\boldsymbol{u}\right)-\kappa\left(\boldsymbol{u}\times\boldsymbol{B},\boldsymbol{J}\right)-\kappa\left(\boldsymbol{J}\times\boldsymbol{B},\boldsymbol{u}\right)\right)=0.\label{eq:SAVenergyq}
\end{equation}
By combining (\ref{eq:SAVenergyu})-(\ref{eq:SAVenergyq}), we derive
(\ref{eq:SAVenergy}).
\end{proof}
\medskip{}

\begin{rem}
In this paper, the scalar auxiliary variable is only a time-dependent
function $q(t)=\exp(-t/T)$ not a energy-related function. With this
treatment, the algebraic equation for the scalar auxiliary variable
is linear and uni-solvent. Moreover, the scalar auxiliary variable of this type admits a general form, $q(t)=C_{q,0}\exp(-C_{q,1}t/T)$
with $C_{q,0}\ne0$ and $C_{q,1}\ge0$. We refer to \citep{Zhang2022}
for more details about this extension.
\end{rem}

\subsection{First-order scheme}

A first-order scheme for solving the system (\ref{eq:SAVmodel}) can
be readily derived by the backward Euler method. For all $n\ge0$,
we compute $\left(\boldsymbol{u}^{n+1},p^{n+1},\boldsymbol{J}^{n+1},\phi^{n+1}\right)$
by solving
\begin{subequations}
\begin{align}
\delta_{t}\boldsymbol{u}^{n+1}-R_{e}^{-1}\Delta\boldsymbol{u}^{n+1}+\nabla p^{n+1}+q^{n+1}\exp\left(\frac{t^{n+1}}{T}\right)\left(\boldsymbol{u}^{n}\cdot\nabla\boldsymbol{u}^{n}-\kappa\boldsymbol{J}^{n}\times\boldsymbol{B}^{n+1}\right) & =\boldsymbol{0},\label{weakh:u}\\
\mathrm{div}\boldsymbol{u}^{n+1} & =0,\label{weakh:divu}\\
\boldsymbol{J}^{n+1}+\nabla\phi^{n+1}-q^{n+1}\exp\left(\frac{t^{n+1}}{T}\right)\left(\boldsymbol{u}^{n}\times\boldsymbol{B}^{n+1}\right) & =\boldsymbol{0},\label{weakh:J}\\
\mathrm{div}\boldsymbol{J}^{n+1} & =0,\label{weakh:divJ}\\
\delta_{t}q^{n+1}+\frac{q^{n+1}}{T}-\exp\left(\frac{t^{n+1}}{T}\right)\left(\left(\boldsymbol{u}^{n}\cdot\nabla\boldsymbol{u}^{n},\boldsymbol{u}^{n+1}\right)-\kappa\left(\boldsymbol{u}^{n}\times\boldsymbol{B}^{n+1},\boldsymbol{J}^{n+1}\right)-\kappa\left(\boldsymbol{J}^{n}\times\boldsymbol{B}^{n+1},\boldsymbol{u}^{n+1}\right)\right) & =0.\label{weakh:q}
\end{align}
\label{eq:SAVEuler}
\end{subequations}

\begin{rem}
\label{rem:init} The init data $\boldsymbol{J}^{0}$ is obtained
by solving (\ref{eq:modelJ})-(\ref{eq:modeldivJ}) at $t=0$. Namely, $\boldsymbol{J}^{0}$ is a part of the solution to 
\begin{subequations}
	\begin{align}
		\boldsymbol{J}^{0}+\nabla\phi^{0}-\boldsymbol{u}^0\times\boldsymbol{B}^{0} & =\boldsymbol{0},\label{Init:J}\\
		\mathrm{div}\boldsymbol{J}^{0} & =0.\label{Init:divJ}
	\end{align}
	\label{eq:Init}
\end{subequations}
A more reliable way is to compute the system at $t = t_1$ using a coupled scheme \citep{Li19} or decoupled scheme \citep{Zhang2020} to obtain $\left(\boldsymbol{u}^{1},p^1,\boldsymbol{J}^1,\phi^1\right)$. Then the proposed scheme is implemented from $t = t_2$ and initialized by the solution at $t = t_1$.
\end{rem}
\medskip{}

\begin{rem}
It is worth noting that the function $\boldsymbol{B}$ is a given external magnetic field in this paper. In the numerical scheme, the symbol $\boldsymbol{B}^{n+1}$ used only indicates that its value may change with time, not that it must be computed within. This remark also applies to the second-order scheme \eqref{eq:SAVAB}.
\end{rem}
\medskip{}

First of all, we prove the unconditionally energy stability of the
scheme as follows.
\begin{thm}
\label{thm:SeEnergyLaws} The scheme (\ref{eq:SAVEuler}) is unconditionally
energy stable in the sense that the following energy estimate
\begin{equation}
\delta_{t}\mathrm{E}_{\text{EL}}^{n+1}\le-R_{e}^{-1}\left\Vert \nabla\boldsymbol{u}^{n+1}\right\Vert ^{2}-\kappa\left\Vert \boldsymbol{J}^{n+1}\right\Vert ^{2}-\frac{1}{T}\left|q^{n+1}\right|^{2}\quad\forall n\ge0,\label{eq:SemiEn}
\end{equation}
 holds, where $\mathrm{E}_{\text{EL}}^{n+1}:=\frac{1}{2}\left\Vert \boldsymbol{u}^{n+1}\right\Vert ^{2}+\frac{1}{2}\left|q^{n+1}\right|^{2}.$
\end{thm}
\begin{proof}
Taking the inner product of (\ref{weakh:u}) with $\boldsymbol{u}^{n+1}$,
and using the identity $2\left(a-b,a\right)=a^{2}-b^{2}+(a-b)^{2},$it
yields,
\begin{align}
\frac{\left\Vert \boldsymbol{u}^{n+1}\right\Vert ^{2}-\left\Vert \boldsymbol{u}^{n}\right\Vert ^{2}+\left\Vert \boldsymbol{u}^{n+1}-\boldsymbol{u}^{n}\right\Vert ^{2}}{2\tau}+R_{e}^{-1}\left\Vert \nabla\boldsymbol{u}^{n+1}\right\Vert ^{2}\nonumber \\
+q^{n+1}\exp\left(\frac{t^{n+1}}{T}\right)\left(\boldsymbol{u}^{n}\cdot\nabla\boldsymbol{u}^{n},\boldsymbol{u}^{n+1}\right)-q^{n+1}\exp\left(\frac{t^{n+1}}{T}\right)\kappa\left(\boldsymbol{J}^{n}\times\boldsymbol{B}^{n+1},\boldsymbol{u}^{n+1}\right) & =0\label{eq:Semiu}
\end{align}
Taking the inner product of (\ref{weakh:J}) with \textbf{$\kappa\boldsymbol{J}^{n+1}$},
we obtain
\begin{equation}
\kappa\left\Vert \boldsymbol{J}^{n+1}\right\Vert ^{2}-q^{n+1}\exp\left(\frac{t^{n+1}}{T}\right)\left(\boldsymbol{u}^{n}\times\boldsymbol{B}^{n+1},\boldsymbol{J}^{n+1}\right)=0.\label{eq:SemiJen}
\end{equation}
Multiplying (\ref{weakh:q}) by $q^{n+1}$ leads to 
\begin{align}
\frac{\left|q^{n+1}\right|^{2}-\left|q^{n}\right|^{2}+\left|q^{n+1}-q^{n}\right|^{2}}{2\tau}+\frac{1}{T}\left|q^{n+1}\right|^{2}\nonumber \\
-q^{n+1}\exp\left(\frac{t^{n+1}}{T}\right)\left(\left(\boldsymbol{u}^{n}\cdot\nabla\boldsymbol{u}^{n},\boldsymbol{u}^{n+1}\right)-\kappa\left(\boldsymbol{u}^{n}\times\boldsymbol{B}^{n+1},\boldsymbol{J}^{n+1}\right)-\kappa\left(\boldsymbol{J}^{n}\times\boldsymbol{B}^{n+1},\boldsymbol{u}^{n+1}\right)\right) & =0\label{eq:Semiq}
\end{align}
By taking the summations of (\ref{eq:Semiu}), (\ref{eq:SemiJen})
and (\ref{eq:Semiq}), we obtain
\begin{align*}
\frac{\left\Vert \boldsymbol{u}^{n+1}\right\Vert ^{2}-\left\Vert \boldsymbol{u}^{n}\right\Vert ^{2}+\left\Vert \boldsymbol{u}^{n+1}-\boldsymbol{u}^{n}\right\Vert ^{2}}{2\tau}+R_{e}^{-1}\left\Vert \nabla\boldsymbol{u}^{n+1}\right\Vert ^{2}\\
+\kappa\left\Vert \boldsymbol{J}^{n+1}\right\Vert ^{2}+\frac{\left|q^{n+1}\right|^{2}-\left|q^{n}\right|^{2}+\left|q^{n+1}-q^{n}\right|^{2}}{2\tau}+\frac{1}{T}\left|q^{n+1}\right|^{2} & =0
\end{align*}
This yields (\ref{eq:SemiEn}).
\end{proof}
\medskip{}

Considering the well stability proved in the previous theorem, we
further elaborate how to implement the proposed schemes in an efficient
way. While the system of equations in (\ref{eq:SAVEuler}) are coupled
with one another, they can be solved in a decoupled fashion, thanks
to the fact that the auxiliary variable $q(t)$ is a scalar number,
not a field function. We next present such a solution algorithm. Let
\begin{equation}
S^{n+1}=q^{n+1}\exp\left(\frac{t^{n+1}}{T}\right).\label{eq:q}
\end{equation}
We rewrite (\ref{weakh:u})\textendash (\ref{weakh:divJ}) into
\[
\begin{aligned}\delta_{t}\boldsymbol{u}-R_{e}^{-1}\Delta\boldsymbol{u}^{n+1}+\nabla p^{n+1} & =-S^{n+1}\left(\boldsymbol{u}^{n}\cdot\nabla\boldsymbol{u}^{n}-\kappa\boldsymbol{J}^{n}\times\boldsymbol{B}^{n+1}\right),\\
\mathrm{div}\boldsymbol{u}^{n+1} & =0,\\
\boldsymbol{J}^{n+1}+\nabla\phi^{n+1} & =S^{n+1}\left(\boldsymbol{u}^{n}\times\boldsymbol{B}^{n+1}\right),\\
\mathrm{div}\boldsymbol{J}^{n+1} & =0.
\end{aligned}
\]
Barring the unknown scalar number $S^{n+1}$, this is a linear equation
with respect to $\left(\boldsymbol{u}^{n+1},p^{n+1},\boldsymbol{J}^{n+1},\phi^{n+1}\right)$.
Inspired by the work in, we define two field functions $\left(\boldsymbol{u}_{i}^{n+1},p_{i}^{n+1},\boldsymbol{J}_{i}^{n+1},\phi_{i}^{n+1}\right)$,
$i=1,2$, as solutions to the following problems:
\begin{equation}
\begin{cases}
\begin{cases}
\frac{\boldsymbol{u}_{1}^{n+1}-\boldsymbol{u}^{n}}{\tau}-R_{e}^{-1}\Delta\boldsymbol{u}_{1}^{n+1}+\nabla p_{1}^{n+1} & =\boldsymbol{0},\\
\mathrm{div}\boldsymbol{u}_{1}^{n+1} & =0,
\end{cases}\\
\begin{cases}
\frac{\boldsymbol{u}_{2}^{n+1}}{\tau}-R_{e}^{-1}\Delta\boldsymbol{u}_{2}^{n+1}+\nabla p_{2}^{n+1} & =\kappa\boldsymbol{J}^{n}\times\boldsymbol{B}^{n+1}-\boldsymbol{u}^{n}\cdot\nabla\boldsymbol{u}^{n},\\
\mathrm{div}\boldsymbol{u}_{2}^{n+1} & =0,
\end{cases}
\end{cases}\label{eq:weakh11}
\end{equation}
 and 
\begin{equation}
\begin{cases}
\begin{cases}
\boldsymbol{J}_{1}^{n+1}+\nabla\phi_{1}^{n+1} & =\boldsymbol{0},\\
\mathrm{div}\boldsymbol{J}_{1}^{n+1} & =0,
\end{cases}\\
\begin{cases}
\boldsymbol{J}_{2}^{n+1}+\nabla\phi_{2}^{n+1} & =\boldsymbol{u}^{n}\times\boldsymbol{B}^{n+1},\\
\mathrm{div}\boldsymbol{J}_{2}^{n+1} & =0.
\end{cases}
\end{cases}\label{eq:weakh12}
\end{equation}
The two systems in (\ref{eq:weakh11}) and (\ref{eq:weakh12}) are
linear with same constant coefficients. We further derive immediately
from the first two equations in (\ref{eq:weakh12}) that $\boldsymbol{J}_{1}^{n+1}=\boldsymbol{0}$
and $\phi_{1}^{n+1}=0$, thus we need not to solve these two variables
faithfully. 

Then it is straightforward to verify that the solution to (\ref{eq:SAVEuler})
is given by
\begin{equation}
\left(\boldsymbol{u}^{n+1},p^{n+1},\boldsymbol{J}^{n+1},\phi^{n+1}\right)=\left(\boldsymbol{u}_{1}^{n+1},p_{1}^{n+1},\boldsymbol{J}_{1}^{n+1},\phi_{1}^{n+1}\right)+S^{n+1}\left(\boldsymbol{u}_{2}^{n+1},p_{2}^{n+1},\boldsymbol{J}_{2}^{n+1},\phi_{2}^{n+1}\right).\label{eq:soldec}
\end{equation}
where $S^{n+1}$ is to be determined. Inserting equation (\ref{eq:soldec})
into equation (\ref{weakh:q}) leads
\begin{equation}
\left(\frac{T+\tau}{T\tau}-\exp\left(\frac{2t^{n+1}}{T}\right)A_{2}\right)\exp\left(-\frac{t^{n+1}}{T}\right)S^{n+1}=\exp\left(\frac{t^{n+1}}{T}\right)A_{1}+\frac{1}{\tau}q^{n},\label{eq:solq}
\end{equation}
where $A_{i},i=1,2$ is defined by
\[
A_{i}=\left(\boldsymbol{u}^{n}\cdot\nabla\boldsymbol{u}^{n},\boldsymbol{u}_{i}^{n+1}\right)-\kappa\left(\boldsymbol{J}^{n}\times\boldsymbol{B}^{n+1},\boldsymbol{u}_{i}^{n+1}\right)-\kappa\left(\boldsymbol{u}^{n}\times\boldsymbol{B}^{n+1},\boldsymbol{J}_{i}^{n+1}\right).
\]
Thus, we arrive at the final solution algorithm. It involves the following
steps:
\begin{enumerate}
\item Solve equations (\ref{eq:weakh11}) and (\ref{eq:weakh12}) for $\left(\boldsymbol{u}_{i}^{n+1},p_{i}^{n+1}\right)$
and $\left(\boldsymbol{J}_{i}^{n+1},\phi_{i}^{n+1}\right)$, $i=1,2$.
\item Solve equation (\ref{eq:solq}) for $S^{n+1}$.
\item Compute $\left(\boldsymbol{u}^{n+1},p^{n+1},\boldsymbol{J}^{n+1},\phi^{n+1}\right)$
by (\ref{eq:soldec}), compute $q^{n+1}$ by (\ref{eq:solq}) and
(\ref{eq:q}).
\end{enumerate}
In summary, at each time step, we only need to solve two generalized
Stokes equations in (\ref{eq:weakh11}) and one Darcy equations
in (\ref{eq:weakh12}) with constant coefficients plus a linear algebraic
equation (\ref{eq:solq}) at each time step. Hence, the scheme is
very efficient in practical calculations.

\medskip{}
\begin{rem}
\label{rem:bdy}The non-homogeneous boundary conditions $\boldsymbol{u}=\boldsymbol{u}_{b},\,\boldsymbol{J}\cdot\boldsymbol{n}=J_{b}\text{ on }\Gamma$
instead of (\ref{eq:bdy}) can be handled by making several modifications.
We only need to slightly modify the SAV variable $q(t)$ to include
the boundary integration as follows, 
\[
\frac{\mathrm{d}q}{\mathrm{dt}}+\frac{q}{T}-\exp\left(\frac{t}{T}\right)\left(\left(\boldsymbol{u}\cdot\nabla\boldsymbol{u},\boldsymbol{u}\right)-\frac{1}{2}\int_{\Gamma}\left(\boldsymbol{u}_{b}\cdot\boldsymbol{n}\right)\left|\boldsymbol{u}_{b}\right|^{2}ds-\kappa\left(\boldsymbol{u}\times\boldsymbol{B}^{n+1},\boldsymbol{J}\right)-\kappa\left(\boldsymbol{J}\times\boldsymbol{B}^{n+1},\boldsymbol{u}\right)\right)=0.
\]
In the decoupled procedures of (\ref{eq:weakh11})-(\ref{eq:weakh12}),
we need to impose the following boundary conditions on $\Gamma$,
$\boldsymbol{u}_{1}^{n+1}=\boldsymbol{u}_{b},\,\boldsymbol{J}_{1}^{n+1}\cdot\boldsymbol{n}=J_{b}$
and $\boldsymbol{u}_{2}^{n+1}=\boldsymbol{0},\,\boldsymbol{J}_{2}^{n+1}\cdot\boldsymbol{n}=0$.
\medskip{}

Finally, we state that the scheme (\ref{eq:SAVEuler}) is well-defined.
That is to say, it is uniquely solvable at each time step.
\end{rem}
\begin{thm}
\label{thm:uni}The scheme (\ref{eq:SAVEuler}) admits a unique solution
at each time step.
\end{thm}
\begin{proof}
From the above-detailed implementation process, we only need to prove
the existence and uniqueness of the solutions to (\ref{eq:weakh11}),
(\ref{eq:weakh12}) and (\ref{eq:solq}). First of all, it is easy
to see that (\ref{eq:weakh11}) and (\ref{eq:weakh12}) are generalized
Stokes problems and Darcy problems, respectively. From the saddle
point theory, they are uniquely solvable. 

To end the proof, we turn to prove the solvability of (\ref{eq:solq})
by verifying $\frac{T+\tau}{T\tau}-\exp\left(\frac{2t^{n+1}}{T}\right)A_{2}\neq0$.
By taking the $\boldsymbol{L}^{2}$-inner product of the third equation
in (\ref{eq:weakh11}) with $\boldsymbol{u}_{2}^{n+1}$ and using
the last equation in (\ref{eq:weakh11}), we have 
\[
-\left(\boldsymbol{u}^{n}\cdot\nabla\boldsymbol{u}^{n},\boldsymbol{u}_{2}^{n+1}\right)+\kappa\left(\boldsymbol{J}^{n}\times\boldsymbol{B}^{n+1},\boldsymbol{u}_{2}^{n+1}\right)=\frac{\left\Vert \boldsymbol{u}_{2}^{n+1}\right\Vert ^{2}}{\tau}+R_{e}^{-1}\left\Vert \nabla\boldsymbol{u}_{2}^{n+1}\right\Vert ^{2}.
\]
By taking the $\boldsymbol{L}^{2}$-inner product of the third equation
in (\ref{eq:weakh12}) with $\kappa\boldsymbol{J}_{2}^{n+1}$ and
using the last equation in (\ref{eq:weakh12}), we get 
\[
\kappa\left(\boldsymbol{u}^{n}\times\boldsymbol{B}^{n+1},\boldsymbol{J}_{2}^{n+1}\right)=\kappa\left\Vert \boldsymbol{J}_{2}^{n+1}\right\Vert ^{2}.
\]
Combing all the estimates above, we conclude 
\begin{align*}
 \frac{T+\tau}{T\tau}-\exp\left(\frac{2t^{n+1}}{T}\right)A_{2} =\frac{T+\tau}{T\tau}+\exp\left(\frac{2t^{n+1}}{T}\right)\left(\frac{\left\Vert \boldsymbol{u}_{2}^{n+1}\right\Vert ^{2}}{\tau}+R_{e}^{-1}\left\Vert \nabla\boldsymbol{u}_{2}^{n+1}\right\Vert ^{2}+\kappa\left\Vert \boldsymbol{J}_{2}^{n+1}\right\Vert ^{2}\right)>0.
\end{align*}
This yields the existence and uniqueness of $S^{n+1}$. Hence, the
scheme (\ref{eq:SAVEuler}) admits a unique solution. The proof is
finished.
\end{proof}

\subsection{Second-order scheme}

A second-order scheme based on backward differential formula (BDF)
for (\ref{eq:SAVmodel}) is constructed as follows. For all $n\ge1$,
we compute $\left(\boldsymbol{u}^{n+1},p^{n+1},\boldsymbol{J}^{n+1},\phi^{n+1}\right)$
by solving
\begin{subequations}
\begin{align}
\delta_{t}^{2}\boldsymbol{u}^{n+1}-\nu\Delta\boldsymbol{u}^{n+1}+\nabla p^{n+1}+q^{n+1}\exp\left(\frac{t^{n+1}}{T}\right)\left(\hat{\boldsymbol{u}}^{n+1}\cdot\nabla\hat{\boldsymbol{u}}^{n+1}-\kappa\hat{\boldsymbol{J}}^{n+1}\times\boldsymbol{B}^{n+1}\right) & =\boldsymbol{0},\label{weakh:ABu}\\
\mathrm{div}\boldsymbol{u}^{n+1} & =0,\label{weakh:ABdivu}\\
\boldsymbol{J}^{n+1}+\nabla\phi^{n+1}-q^{n+1}\exp\left(\frac{t^{n+1}}{T}\right)\left(\hat{\boldsymbol{u}}^{n+1}\times\boldsymbol{B}^{n+1}\right) & =\boldsymbol{0},\label{weakh:ABJ}\\
\mathrm{div}\boldsymbol{J}^{n+1} & =0,\label{weakh:ABdivJ}\\
\delta_{t}q^{n+1}+\frac{q^{n+1}}{T}\nonumber \\
-\exp\left(\frac{t^{n+1}}{T}\right)\left(\left(\hat{\boldsymbol{u}}^{n+1}\cdot\nabla\hat{\boldsymbol{u}}^{n+1},\boldsymbol{u}^{n+1}\right)-\kappa\left(\hat{\boldsymbol{u}}^{n+1}\times\boldsymbol{B}^{n+1},\boldsymbol{J}^{n+1}\right)-\kappa\left(\hat{\boldsymbol{J}}^{n+1}\times\boldsymbol{B}^{n+1},\boldsymbol{u}^{n+1}\right)\right) & =0.\label{weakh:ABq}
\end{align}
\label{eq:SAVAB}
\end{subequations}
From (\ref{eq:BDF0}), it is easy to see that when $n=0$, $\left(\boldsymbol{u}^{1},p^{1},\boldsymbol{J}^{1},\phi^{1},q^{1}\right)$
is computed by the first-order scheme described in (\ref{eq:SAVEuler}).
Now we derive the energy stability of the above scheme.
\begin{thm}
\label{thm:SeEnergyLawsAB} The scheme (\ref{eq:SAVAB}) is unconditionally
energy stable in the sense that the following energy estimate
\begin{equation}
\delta_{t}\mathrm{E}_{\text{BDF}}^{n+1}\le-R_{e}^{-1}\left\Vert \nabla\boldsymbol{u}^{n+1}\right\Vert ^{2}-\kappa\left\Vert \boldsymbol{J}^{n+1}\right\Vert ^{2}-\frac{1}{T}\left|q^{n+1}\right|^{2}\quad\forall n\ge0,\label{eq:SemiEnAB}
\end{equation}
 holds, where $\mathrm{E}_{\text{BDF}}^{n+1}:=\frac{1}{4}\left(\left\Vert \boldsymbol{u}^{n+1}\right\Vert ^{2}+\left\Vert 2\boldsymbol{u}^{n+1}-\boldsymbol{u}^{n}\right\Vert ^{2}\right)+\frac{1}{4}\left(\left\Vert q^{n+1}\right\Vert ^{2}+\left\Vert 2q^{n+1}-q^{n}\right\Vert ^{2}\right)$.
\end{thm}
\begin{proof}
Taking the inner product of (\ref{weakh:ABu}) with $\boldsymbol{u}^{n+1}$,
and using the identity 
\[
2(3a-4b+c,a)=|a|^{2}-|b|^{2}+|2a-b|^{2}-|2b-c|^{2}+|a-2b+c|^{2},
\]
it yields,
\begin{align}
\frac{1}{4\tau}\left(\left\Vert \boldsymbol{u}^{n+1}\right\Vert ^{2}+\left\Vert 2\boldsymbol{u}^{n+1}-\boldsymbol{u}^{n}\right\Vert ^{2}-\left\Vert \boldsymbol{u}^{n}\right\Vert ^{2}-\left\Vert 2\boldsymbol{u}^{n}-\boldsymbol{u}^{n-1}\right\Vert ^{2}+\left\Vert \boldsymbol{u}^{n+1}-2\boldsymbol{u}^{n}+\boldsymbol{u}^{n-1}\right\Vert ^{2}\right)\nonumber \\
+R_{e}^{-1}\left\Vert \nabla\boldsymbol{u}^{n+1}\right\Vert ^{2}+q^{n+1}\exp\left(\frac{t^{n+1}}{T}\right)\left(\hat{\boldsymbol{u}}^{n+1}\cdot\nabla\hat{\boldsymbol{u}}^{n+1},\boldsymbol{u}^{n+1}\right)-q^{n+1}\exp\left(\frac{t^{n+1}}{T}\right)\kappa\left(\hat{\boldsymbol{J}}^{n+1}\times\boldsymbol{B}^{n+1},\boldsymbol{u}^{nq}\right)=0.\label{eq:SemiuAB}
\end{align}
Taking the inner product of (\ref{weakh:ABJ}) with $\kappa\boldsymbol{J}^{n+1}$,
we have
\begin{equation}
\kappa\left\Vert \boldsymbol{J}^{n+1}\right\Vert ^{2}-q^{n+1}\exp\left(\frac{t^{n+1}}{T}\right)\left(\hat{\boldsymbol{u}}^{n+1}\times\boldsymbol{B}^{n+1},\boldsymbol{J}^{n+1}\right)=0.\label{eq:SemiJAB}
\end{equation}
Multiplying (\ref{weakh:ABq}) by $q^{n+1}$ leads to 
\begin{align}
\frac{1}{4\tau}\left(\left\Vert q^{n+1}\right\Vert ^{2}+\left\Vert 2q^{n+1}-q^{n}\right\Vert ^{2}-\left\Vert q^{n}\right\Vert ^{2}-\left\Vert 2q^{n}-q^{n-1}\right\Vert ^{2}+\left\Vert q^{n+1}-2q^{n}+q^{n-1}\right\Vert ^{2}\right)\nonumber \\
+\frac{1}{T}\left|q^{n+1}\right|^{2}-q^{n+1}\exp\left(\frac{t^{n+1}}{T}\right)\left(\hat{\boldsymbol{u}}^{n+1}\cdot\nabla\hat{\boldsymbol{u}}^{n+1},\boldsymbol{u}^{n+1}\right)\nonumber \\
+q^{n+1}\exp\left(\frac{t^{n+1}}{T}\right)\kappa\left(\hat{\boldsymbol{u}}^{n+1}\times\boldsymbol{B}^{n+1},\boldsymbol{J}^{n+1}\right)+q^{n+1}\exp\left(\frac{t^{n+1}}{T}\right)\kappa\left(\hat{\boldsymbol{J}}^{n+1}\times\boldsymbol{B}^{n+1},\boldsymbol{u}^{n+1}\right)=0.\label{eq:SemiqAB}
\end{align}
By taking the summations of (\ref{eq:SemiuAB}) and (\ref{eq:SemiqAB}),
we obtain
\begin{align*}
\frac{1}{4\tau}\left(\left\Vert \boldsymbol{u}^{n+1}\right\Vert ^{2}+\left\Vert 2\boldsymbol{u}^{n+1}-\boldsymbol{u}^{n}\right\Vert ^{2}-\left\Vert \boldsymbol{u}^{n}\right\Vert ^{2}-\left\Vert 2\boldsymbol{u}^{n}-\boldsymbol{u}^{n-1}\right\Vert ^{2}+\left\Vert \boldsymbol{u}^{n+1}-2\boldsymbol{u}^{n}+\boldsymbol{u}^{n-1}\right\Vert ^{2}\right)\\
+\frac{1}{4\tau}\left(\left\Vert q^{n+1}\right\Vert ^{2}+\left\Vert 2q^{n+1}-q^{n}\right\Vert ^{2}-\left\Vert q^{n}\right\Vert ^{2}-\left\Vert 2q^{n}-q^{n-1}\right\Vert ^{2}+\left\Vert q^{n+1}-2q^{n}+q^{n-1}\right\Vert ^{2}\right)\\
+R_{e}^{-1}\left\Vert \nabla\boldsymbol{u}^{n+1}\right\Vert ^{2}+\kappa\left\Vert \boldsymbol{J}^{n+1}\right\Vert ^{2}+\frac{1}{T}\left|q^{n+1}\right|^{2}=0,
\end{align*}
which completes the proof.
\end{proof}
\smallskip{}

The second-order scheme (\ref{eq:SAVAB}) can be implemented efficiently
in the same way as the first-scheme (\ref{eq:SAVEuler}). For the
convenience of the readers, we present the final solution algorithm.
For $n=0$, we have discussed the implementation in the previous
subsection. For $n\ge1$, it involves the following steps:
\begin{enumerate}
\item Get the solutions $\left(\boldsymbol{u}_{i}^{n+1},p_{i}^{n+1},\boldsymbol{J}_{i}^{n+1},\phi_{i}^{n+1}\right)$,
$i=1,2$.
\begin{enumerate}
\item Get the solutions $\left(\boldsymbol{u}_{i}^{n+1},p_{i}^{n+1}\right)$,
$i=1,2$, by solving 
\begin{equation}
\begin{cases}
\frac{3\boldsymbol{u}_{1}^{n+1}-4\boldsymbol{u}^{n}+\boldsymbol{u}^{n-1}}{2\tau}-R_{e}^{-1}\Delta\boldsymbol{u}_{1}^{n+1}+\nabla p_{1}^{n+1} & =\boldsymbol{0},\\
\mathrm{div}\boldsymbol{u}_{1}^{n+1} & =0,
\end{cases}\label{eq:weakhu11AB}
\end{equation}
and 
\begin{equation}
\begin{cases}
\frac{\boldsymbol{u}_{2}^{n+1}}{\tau}-R_{e}^{-1}\Delta\boldsymbol{u}_{2}^{n+1}+\nabla p_{2}^{n+1} & =\kappa\hat{\boldsymbol{J}}^{n+1}\times\boldsymbol{B}^{n+1}-\hat{\boldsymbol{u}}^{n+1}\cdot\nabla\hat{\boldsymbol{u}}^{n+1},\\
\mathrm{div}\boldsymbol{u}_{2}^{n+1} & =0.
\end{cases}\label{eq:weakhu12AB}
\end{equation}
\item Get the solutions $\left(\boldsymbol{J}_{i}^{n+1},\phi_{i}^{n+1}\right)$,
$i=1,2$, by solving 
\begin{equation}
\begin{cases}
\boldsymbol{J}_{1}^{n+1}+\nabla\phi_{1}^{n+1} & =\boldsymbol{0},\\
\mathrm{div}\boldsymbol{J}_{1}^{n+1} & =0,
\end{cases}\label{eq:weakhJ11AB}
\end{equation}
 and 
\begin{equation}
\begin{cases}
\boldsymbol{J}_{2}^{n+1}+\nabla\phi_{2}^{n+1} & =\hat{\boldsymbol{u}}^{n+1}\times\boldsymbol{B}^{n+1},\\
\mathrm{div}\boldsymbol{J}_{2}^{n+1} & =0.
\end{cases}\label{eq:weakhJ12AB}
\end{equation}
\end{enumerate}
\item Get the solution $S^{n+1}$ by solving 
\begin{equation}
\left(\frac{3T+2\tau}{2\tau T}-\exp\left(\frac{2t^{n+1}}{T}\right)A_{2}\right)\exp\left(-\frac{t^{n+1}}{T}\right)S^{n+1}=\exp\left(\frac{t^{n+1}}{T}\right)A_{1}+\frac{1}{2\tau}\left(4q^{n}-q^{n-1}\right),\label{eq:weakh11q}
\end{equation}
where $A_{i},i=1,2$ is defined by
\[
A_{i}=\left(\hat{\boldsymbol{u}}^{n+1}\cdot\nabla\hat{\boldsymbol{u}}^{n+1},\boldsymbol{u}_{i}^{n+1}\right)-\kappa\left(\hat{\boldsymbol{J}}^{n+1}\times\boldsymbol{B}^{n+1},\boldsymbol{u}_{i}^{n+1}\right)-\kappa\left(\hat{\boldsymbol{u}}^{n+1}\times\boldsymbol{B}^{n+1},\boldsymbol{J}_{i}^{n+1}\right).
\]
\item Compute $\left(\boldsymbol{u}^{n+1},p^{n+1},\boldsymbol{J}^{n+1},\phi^{n+1}\right)$
by 
\[
\left(\boldsymbol{u}^{n+1},p^{n+1},\boldsymbol{J}^{n+1},\phi^{n+1}\right)=\left(\boldsymbol{u}_{1}^{n+1},p_{1}^{n+1},\boldsymbol{J}_{1}^{n+1},\phi_{1}^{n+1}\right)+S^{n+1}\left(\boldsymbol{u}_{2}^{n+1},p_{2}^{n+1},\boldsymbol{J}_{2}^{n+1},\phi_{2}^{n+1}\right),
\]
 and compute $q^{n+1}$ by $q^{n+1}=S^{n+1}\exp\left(-t^{n+1}/T\right)$.
\end{enumerate}
Similar to the implementation of the first-order scheme, we find that the solutions to (\ref{eq:weakhJ11AB}) are $\boldsymbol{J}_{1}^{n+1}=\boldsymbol{0}$
and $\phi_{1}^{n+1}=0$, which means that they do not need to be solved veritably. Therefore, the second-order scheme can be efficiently implemented as the first-order scheme by solving a sequence of linear systems with constant coefficients. 
\medskip{}

By using exactly the same procedure as Theorem \ref{thm:uni} for
the first-order scheme (\ref{eq:SAVEuler}), we can show that the
second-order scheme (\ref{eq:SAVAB}) is uniquely solvable.
\begin{thm}
\label{thm:uniAB}The scheme (\ref{eq:SAVAB}) admits a unique solution
at each time step.
\end{thm}
\begin{proof}
Since the proof are similar to the one for Theorem \ref{thm:uni},
we omit the details.
\end{proof}

\section{Error analysis\label{sec:Error}}

In this section, we derive the error estimates for the first-order
scheme. Similar analysis can also be carried out for the second-order
scheme by combing the procedures below but the detail is much more
tedious. We also emphasize that while both schemes can be used in
the three-dimensional case, the error analysis can not be easily extended
to the three-dimension case due to some technical issues. Hence, we
set $d=2$ in this section.

To do this, we denote the error functions as
\[
e_{\boldsymbol{u}}^{n}=\boldsymbol{u}^{n}-\boldsymbol{u}\left(t^{n}\right),\,e_{p}^{n}=p^{n}-p\left(t^{n}\right),\,e_{\boldsymbol{J}}^{n}=\boldsymbol{J}^{n}-\boldsymbol{J}\left(t^{n}\right),\,e_{\phi}^{n}=\phi^{n}-\phi\left(t^{n}\right),\,e_{q}^{n}=q^{n}-q\left(t^{n}\right).
\]
Subtracting (\ref{modeleq:u}) at $t^{n+1}$ from (\ref{eq:SAVEuler}),
we obtain the following error equations,
\begin{align}
\delta_{t}e_{\boldsymbol{u}}^{n+1}-R_{e}^{-1}\Delta e_{\boldsymbol{u}}^{n+1}+\nabla e_{p}^{n+1}-\exp\left(\frac{t^{n+1}}{T}\right)\left(q\left(t^{n+1}\right)\boldsymbol{u}\left(t^{n+1}\right)\cdot\nabla\boldsymbol{u}\left(t^{n+1}\right)-q^{n+1}\boldsymbol{u}^{n}\cdot\nabla\boldsymbol{u}^{n}\right)\nonumber \\
=R_{\boldsymbol{u}}^{n+1}+\kappa\exp\left(\frac{t^{n+1}}{T}\right)\left(q\left(t^{n+1}\right)\boldsymbol{J}^{n}\times\boldsymbol{B}^{n+1}-q^{n+1}\boldsymbol{J}\left(t^{n+1}\right)\times\boldsymbol{B}^{n+1}\right),\label{eq:erroru}\\
\nabla\cdot e_{\boldsymbol{u}}^{n+1}=0,\label{eq:errorp}\\
e_{\boldsymbol{J}}^{n+1}+\nabla e_{\phi}^{n+1}=\exp\left(\frac{t^{n+1}}{T}\right)\left(q^{n+1}\boldsymbol{u}^{n}\times\boldsymbol{B}^{n+1}-q\left(t^{n+1}\right)\boldsymbol{u}\left(t^{n+1}\right)\times\boldsymbol{B}^{n+1}\right),\label{eq:errorJ}\\
\delta_{t}e_{q}^{n+1}+\frac{1}{T}e_{q}^{n+1}-\exp\left(\frac{t^{n+1}}{T}\right)\left(\left(\boldsymbol{u}^{n}\cdot\nabla\boldsymbol{u}^{n},\boldsymbol{u}^{n+1}\right)-\left(\boldsymbol{u}\left(t^{n+1}\right)\cdot\nabla\boldsymbol{u}\left(t^{n+1}\right),\boldsymbol{u}\left(t^{n+1}\right)\right)\right)\nonumber \\
=R_{q}^{n+1}-\kappa\exp\left(\frac{t^{n+1}}{T}\right)\left(\left(\boldsymbol{J}^{n}\times\boldsymbol{B}^{n+1},\boldsymbol{u}^{n+1}\right)-\left(\boldsymbol{J}\left(t^{n+1}\right)\times\boldsymbol{B}^{n+1},\boldsymbol{u}\left(t^{n+1}\right)\right)\right)\nonumber \\
-\kappa\exp\left(\frac{t^{n+1}}{T}\right)\left(\left(\boldsymbol{u}^{n}\times\boldsymbol{B}^{n+1},\boldsymbol{J}^{n+1}\right)-\left(\boldsymbol{u}\left(t^{n+1}\right)\times\boldsymbol{B}^{n+1},\boldsymbol{J}\left(t^{n+1}\right)\right)\right),\label{eq:errorq}
\end{align}
where the truncation errors are defined by 
\[
R_{\boldsymbol{u}}^{n+1}=\frac{1}{\tau}\int_{t^{n}}^{t^{n+1}}\left(t^{n}-t\right)\boldsymbol{u}_{tt}{\rm d}t,\quad R_{q}^{n+1}=\frac{1}{\tau}\int_{t^{n}}^{t^{n+1}}\left(t^{n}-t\right)q_{tt}{\rm d}t.
\]

\begin{lem}[Stability]
 Let $\left(\boldsymbol{u}^{n},p^{n},\boldsymbol{J}^{n},\phi^{n},q^{n}\right),\,n\ge0$
solve (\ref{eq:SAVEuler}). Then it satisfies the following stability
estimate for any $m\ge0$,
\begin{align}
\left\Vert \boldsymbol{u}^{m}\right\Vert ^{2}+\left|q^{m}\right|^{2} & \le k_{1},\label{eq:u_uniform}\\
\tau\sum_{n=0}^{m}\left(R_{e}^{-1}\left\Vert \nabla\boldsymbol{u}^{n}\right\Vert ^{2}+\kappa\left\Vert \boldsymbol{J}^{n}\right\Vert ^{2}+\frac{1}{T}|q^{n}|^{2}\right) & \le k_{2},\label{eq:gradu_uniform}
\end{align}
where the constants $k_{i}$ $(i=1,2)$ are independent of $\tau$. 
\end{lem}
\begin{proof}
By Theorem \ref{thm:SeEnergyLaws}, summing up inequality (\ref{eq:SemiEn})
from $n=0$ to $m-1$, we obtain 
\[
\left\Vert \boldsymbol{u}^{m}\right\Vert ^{2}+\left|q^{m}\right|^{2}+2\tau\sum_{n=0}^{m}\left(R_{e}^{-1}\left\Vert \nabla\boldsymbol{u}^{n+1}\right\Vert ^{2}+\kappa\left\Vert \boldsymbol{J}^{n+1}\right\Vert ^{2}+\frac{1}{T}|q^{n+1}|^{2}\right)\le\left\Vert \boldsymbol{u}^{0}\right\Vert ^{2}+\left\Vert q^{0}\right\Vert ^{2}.
\]
This implies the desired result.
\end{proof}
\medskip{}

Let $P$ be the orthogonal projector in $\boldsymbol{L}^{2}(\Omega)$
onto $\boldsymbol{V}$, we define the Stokes operator $A$ by 
\[
A\boldsymbol{u}=-P\Delta\boldsymbol{u},\quad\forall\boldsymbol{u}\in D(A)=\boldsymbol{H}^{2}(\Omega)\cap\boldsymbol{X}.
\]
The following estimates for the trilinear form $b(\cdot,\cdot,\cdot)$
will be used in our error analysis \citep{Li2020b,Li2021,Temam1995,Layton1998}.
\begin{lem}
The following estimates of the trilinear form holds for $s>1/2$,
\begin{align}
b(\boldsymbol{u},\boldsymbol{v},\boldsymbol{w}) & \le C_{b,0}\left\Vert \nabla\boldsymbol{u}\right\Vert \left\Vert \nabla\boldsymbol{v}\right\Vert \left\Vert \nabla\boldsymbol{w}\right\Vert \quad\forall\boldsymbol{u},\boldsymbol{v},\boldsymbol{w}\in\boldsymbol{X},\label{eq:e_trilinear_form}\\
b(\boldsymbol{u},\boldsymbol{v},\boldsymbol{w}) & \le C_{b,1}\left\Vert \boldsymbol{u}\right\Vert \left\Vert \boldsymbol{v}\right\Vert _{1+s}\left\Vert \nabla\boldsymbol{w}\right\Vert \quad\forall\boldsymbol{u},\boldsymbol{w}\in\boldsymbol{X},\quad\boldsymbol{v}\in\boldsymbol{X}\cap\boldsymbol{H}^{1+s}\left(\Omega\right),\label{eq:e_trilinear_form0}\\
b(\boldsymbol{u},\boldsymbol{v},\boldsymbol{w}) & \le C_{b,2}\left\Vert \boldsymbol{u}\right\Vert _{1+s}\left\Vert \boldsymbol{v}\right\Vert \left\Vert \nabla\boldsymbol{w}\right\Vert \quad\forall\boldsymbol{v},\boldsymbol{w}\in\boldsymbol{X},\quad\boldsymbol{u}\in\boldsymbol{V}\cap\boldsymbol{H}^{1+s}\left(\Omega\right),\label{eq:e_trilinear_form1}\\
b(\boldsymbol{u},\boldsymbol{v},\boldsymbol{w}) & \le C_{b,3}\left\Vert \nabla\boldsymbol{u}\right\Vert \left\Vert \boldsymbol{v}\right\Vert \left\Vert \boldsymbol{w}\right\Vert _{1+s}\quad\forall\boldsymbol{u}\in\boldsymbol{V},\quad\boldsymbol{v}\in\boldsymbol{X},\quad\boldsymbol{w}\in\boldsymbol{X}\cap\boldsymbol{H}^{1+s}\left(\Omega\right),\label{eq:e_trilinear_form2}\\
b(\boldsymbol{u},\boldsymbol{v},\boldsymbol{w}) & \le C_{b,4}\left\Vert \boldsymbol{u}\right\Vert \left\Vert \nabla\boldsymbol{v}\right\Vert \left\Vert \boldsymbol{w}\right\Vert _{1+s}\quad\forall\boldsymbol{u},\boldsymbol{v}\in\boldsymbol{X},\quad\boldsymbol{w}\in\boldsymbol{X}\cap\boldsymbol{H}^{1+s}\left(\Omega\right),\label{eq:e_trilinear_form3}\\
b(\boldsymbol{u},\boldsymbol{v},\boldsymbol{w}) & \le C_{b,5}\left\Vert \nabla\boldsymbol{u}\right\Vert \left\Vert \boldsymbol{v}\right\Vert _{1+s}\left\Vert \boldsymbol{w}\right\Vert \quad\forall\boldsymbol{u},\boldsymbol{w}\in\boldsymbol{X},\quad\boldsymbol{v}\in\boldsymbol{X}\cap\boldsymbol{H}^{1+s}\left(\Omega\right).\label{eq:e_trilinear_form4}
\end{align}
Moreover, for $d=2$, we have
\begin{align}
b(\boldsymbol{u},\boldsymbol{v},\boldsymbol{w}) & \leq C_{b,6}\left\Vert \nabla\boldsymbol{u}\right\Vert ^{1/2}\left\Vert \boldsymbol{u}\right\Vert ^{1/2}\left\Vert \nabla\boldsymbol{v}\right\Vert ^{1/2}\left\Vert \boldsymbol{v}\right\Vert ^{1/2}\left\Vert \nabla\boldsymbol{w}\right\Vert \quad\forall\boldsymbol{u}\in\boldsymbol{V},\quad\boldsymbol{v},\boldsymbol{w}\in\boldsymbol{X},\label{eq:e_trilinear_form2d}\\
b(\boldsymbol{u},\boldsymbol{v},\boldsymbol{w}) & \leq C_{b,7}\left\Vert \nabla\boldsymbol{u}\right\Vert ^{1/2}\left\Vert \boldsymbol{u}\right\Vert ^{1/2}\left\Vert \nabla\boldsymbol{v}\right\Vert ^{1/2}\left\Vert A\boldsymbol{v}\right\Vert ^{1/2}\left\Vert \boldsymbol{w}\right\Vert \quad\forall\boldsymbol{v}\in\boldsymbol{X}\cap\boldsymbol{H}^{2}\left(\Omega\right),\quad\boldsymbol{u},\boldsymbol{w}\in\boldsymbol{X},\label{eq:e_trilinear_form2d1}\\
b(\boldsymbol{u},\boldsymbol{v},\boldsymbol{w}) & \leq C_{b,8}\left\Vert \boldsymbol{u}\right\Vert ^{1/2}\left\Vert A\boldsymbol{u}\right\Vert ^{1/2}\left\Vert \nabla\boldsymbol{v}\right\Vert \left\Vert \boldsymbol{w}\right\Vert \quad\forall\boldsymbol{u}\in\boldsymbol{X}\cap\boldsymbol{H}^{2}\left(\Omega\right),\quad\boldsymbol{v},\boldsymbol{w}\in\boldsymbol{X}.\label{eq:e_trilinear_form2d2}
\end{align}
\end{lem}
\medskip{}

We will frequently use the following discrete version of the Gronwall
lemma \citep{Heywood1990}.
\begin{lem}
\label{lem: gronwall2} Let $a_{n},b_{n},c_{n}$, and $d_{n}$ be
four non-negative sequences satisfying 
\[
a_{m}+\tau\sum_{n=1}^{m}b_{n}\leq\tau\sum_{n=0}^{m-1}a_{n}d_{n}+\tau\sum_{n=0}^{m-1}c_{n}+C,\quad m\geq1,
\]
 where $C$ and $\tau$ are two positive constants. Then
\[
a_{m}+\tau\sum_{n=1}^{m}b_{n}\leq\exp\left(\tau\sum_{n=0}^{m-1}d_{n}\right)\left(\tau\sum_{n=0}^{m-1}c_{n}+C\right),\quad m\geq1.
\]
\end{lem}

\subsection{Error estimates for the velocity and current density}

In this subsection, we derive the following error estimates for the
velocity $\boldsymbol{u}$ and current density $\boldsymbol{J}$. 
\begin{thm}
\label{thm:erroruJq} Assuming $\boldsymbol{u}\in H^{2}(0,T;\boldsymbol{H}^{-1}(\Omega))\bigcap H^{1}(0,T;\boldsymbol{H}^{1+s}(\Omega))\bigcap L^{\infty}(0,T;\boldsymbol{H}^{1+s}(\Omega))$,
$\boldsymbol{J}\in L^{\infty}(0,T;\boldsymbol{L}^{2}(\Omega))\cap H^{1}(0,T;\boldsymbol{L}^{2}(\Omega))$ and $\boldsymbol{B}\in L^{\infty}(0,T;\boldsymbol{L}^{3}(\Omega))$,
then for the scheme (\ref{eq:SAVEuler}), the following error estimate
holds for $m\ge0$,
\begin{align}
 & \left\Vert e_{\boldsymbol{u}}^{m+1}\right\Vert ^{2}+|e_{q}^{m+1}|^{2}+\nu\tau\sum\limits _{n=0}^{m}\left\Vert \nabla e_{\boldsymbol{u}}^{n+1}\right\Vert ^{2}+\kappa\tau\sum\limits _{n=0}^{m}\left\Vert \nabla e_{\boldsymbol{J}}^{n+1}\right\Vert ^{2}+\tau\sum\limits _{n=0}^{m}|e_{q}^{n+1}|^{2}\nonumber \\
 & +\sum\limits _{n=0}^{m}\left\Vert e_{\boldsymbol{u}}^{n+1}-e_{\boldsymbol{u}}^{n}\right\Vert ^{2}+\sum\limits _{n=0}^{m}|e_{q}^{n+1}-e_{q}^{n}|^{2}\leq C\tau^{2}.\label{eq:estuJq}
\end{align}
\end{thm}
\begin{rem}
In Theorem \ref{thm:erroruJq}, the regularity assumption for the
exact solutions are needed to deliver the error estimates. Compared
with the existing works, we have lowered the regularity index $s=1$
to $s>1/2$, which is a small improvement. The regularity for the
exact solutions in is relative weak and may be a reasonable hypothesis
in such a setting. 
\end{rem}
\medskip{}

The proof of the above theorem will be carried out with a sequence
of lemmas below.
\begin{lem}
\label{lem:erroru}Under the assumptions of Theorem \ref{thm:erroruJq},
the following error estimate holds for $\ 0\leq n\leq N-1$,
\begin{align}
 & \frac{\left\Vert e_{\boldsymbol{u}}^{n+1}\right\Vert ^{2}-\left\Vert e_{\boldsymbol{u}}^{n}\right\Vert ^{2}+\left\Vert e_{\boldsymbol{u}}^{n+1}-e_{\boldsymbol{u}}^{n}\right\Vert ^{2}}{2\tau}+\frac{R_{e}^{-1}}{2}\left\Vert \nabla e_{\boldsymbol{u}}^{n+1}\right\Vert ^{2}\nonumber \\
 & \leq\exp\left(\frac{t^{n+1}}{T}\right)e_{q}^{n+1}\left(\kappa\left(\boldsymbol{J}^{n}\times\boldsymbol{B}^{n+1},e_{\boldsymbol{u}}^{n+1}\right)-\left(\boldsymbol{u}^{n}\cdot\nabla\boldsymbol{u}^{n},e_{\boldsymbol{u}}^{n+1}\right)\right)\nonumber \\
 & \quad+C\left(\left\Vert \boldsymbol{u}(t^{n})\right\Vert _{1+s}^{2}+\left\Vert \boldsymbol{u}(t^{n+1})\right\Vert _{1+s}^{2}+\left\Vert \nabla e_{\boldsymbol{u}}^{n}\right\Vert ^{2}\right)\left\Vert e_{\boldsymbol{u}}^{n}\right\Vert ^{2}+C\left\Vert \boldsymbol{B}^{n+1}\right\Vert _{0,3}^{2}\left\Vert e_{\boldsymbol{J}}^{n}\right\Vert ^{2}\nonumber \\
 & \quad+C\tau\int_{t^{n}}^{t^{n+1}}\left\Vert \boldsymbol{u}_{tt}(t)\right\Vert _{-1}^{2}dt+C\tau\left\Vert \boldsymbol{u}(t^{n+1})\right\Vert _{1+s}^{2}\int_{t^{n}}^{t^{n+1}}\left\Vert \boldsymbol{u}_{t}(t)\right\Vert ^{2}dt\nonumber \\
 & \quad+C\tau\left\Vert \boldsymbol{u}^{n}\right\Vert \int_{t^{n}}^{t^{n+1}}\left\Vert \boldsymbol{u}_{t}(t)\right\Vert _{1+s}^{2}dt+C\tau\left\Vert \boldsymbol{B}^{n+1}\right\Vert _{0,3}^{2}\int_{t^{n}}^{t^{n+1}}\left\Vert \boldsymbol{J}_{t}\right\Vert ^{2}{\rm d}t.\label{eq:error_u}
\end{align}
\end{lem}
\begin{proof}
\noindent Taking the $\boldsymbol{L}^{2}$-inner product of (\ref{eq:erroru})
with $e_{\boldsymbol{u}}^{n+1}$ and using (\ref{eq:errorp}), we
get
\begin{align}
 & \frac{\left\Vert e_{\boldsymbol{u}}^{n+1}\right\Vert ^{2}-\left\Vert e_{\boldsymbol{u}}^{n}\right\Vert ^{2}+\left\Vert e_{\boldsymbol{u}}^{n+1}-e_{\boldsymbol{u}}^{n}\right\Vert ^{2}}{2\tau}+R_{e}^{-1}\left\Vert \nabla e_{\boldsymbol{u}}^{n+1}\right\Vert ^{2}\nonumber \\
 & =\left(R_{\boldsymbol{u}}^{n+1},e_{\boldsymbol{u}}^{n+1}\right)+\exp\left(\frac{t^{n+1}}{T}\right)\left(q\left(t^{n+1}\right)\boldsymbol{u}\left(t^{n+1}\right)\cdot\nabla\boldsymbol{u}\left(t^{n+1}\right)-q^{n+1}\boldsymbol{u}^{n}\cdot\nabla\boldsymbol{u}^{n},e_{\boldsymbol{u}}^{n+1}\right)\nonumber \\
 & \quad+\kappa\exp\left(\frac{t^{n+1}}{T}\right)\left(q^{n+1}\boldsymbol{J}^{n}\times\boldsymbol{B}^{n+1}-q\left(t^{n+1}\right)\boldsymbol{J}\left(t^{n+1}\right)\times\boldsymbol{B}^{n+1},e_{\boldsymbol{u}}^{n+1}\right)\nonumber \\
 & \coloneqq\sum_{i=1}^{3}{\rm I}_{i}\label{eq:erroruI}
\end{align}
For term ${\rm I}_{1}$, we use the Young inequality to have 
\begin{equation}
{\rm I}_{1}\le\frac{R_{e}^{-1}}{8}\left\Vert \nabla e_{\boldsymbol{u}}^{n+1}\right\Vert ^{2}+C\tau\int_{t^{n}}^{t^{n+1}}\left\Vert \boldsymbol{u}_{tt}\right\Vert _{-1}^{2}{\rm d}t.\label{eq:truncu}
\end{equation}
For term ${\rm I}_{2}$, we rearrange it as follows,
\begin{align}
{\rm I}_{2} & =\left(\left(\boldsymbol{u}\left(t^{n+1}\right)-\boldsymbol{u}^{n}\right)\cdot\nabla\boldsymbol{u}\left(t^{n+1}\right),e_{\boldsymbol{u}}^{n+1}\right)+\left(\boldsymbol{u}^{n}\cdot\nabla\left(\boldsymbol{u}\left(t^{n+1}\right)-\boldsymbol{u}^{n}\right),e_{\boldsymbol{u}}^{n+1}\right)\nonumber \\
 & \quad-\exp\left(\frac{t^{n+1}}{T}\right)e_{q}^{n+1}\left(\boldsymbol{u}^{n}\cdot\nabla\boldsymbol{u}^{n},e_{\boldsymbol{u}}^{n+1}\right)\nonumber \\
 & ={\rm I_{2,1}}+{\rm I}_{2,2}-\exp\left(\frac{t^{n+1}}{T}\right)e_{q}^{n+1}\left(\boldsymbol{u}^{n}\cdot\nabla\boldsymbol{u}^{n},e_{\boldsymbol{u}}^{n+1}\right).\label{eq:erroruII}
\end{align}
For term ${\rm I}_{2,1}$, it can be bounded by using (\ref{eq:e_trilinear_form0})
and Young inequality,
\begin{align}
{\rm I_{2,1}} & \le C_{b,1}\left\Vert \boldsymbol{u}\left(t^{n+1}\right)-\boldsymbol{u}^{n}\right\Vert \left\Vert \boldsymbol{u}\left(t^{n+1}\right)\right\Vert _{1+s}\left\Vert \nabla e_{\boldsymbol{u}}^{n+1}\right\Vert \nonumber \\
 & \le C_{b,1}\left\Vert e_{\boldsymbol{u}}^{n}\right\Vert \left\Vert \boldsymbol{u}\left(t^{n+1}\right)\right\Vert _{1+s}\left\Vert \nabla e_{\boldsymbol{u}}^{n+1}\right\Vert +C_{b,1}\left\Vert \int_{t^{n}}^{t^{n+1}}\boldsymbol{u}_{t}{\rm d}t\right\Vert \left\Vert \boldsymbol{u}\left(t^{n+1}\right)\right\Vert _{1+s}\left\Vert \nabla e_{\boldsymbol{u}}^{n+1}\right\Vert \nonumber \\
 & \le\frac{R_{e}^{-1}}{8}\left\Vert \nabla e_{\boldsymbol{u}}^{n+1}\right\Vert ^{2}+C\left\Vert \boldsymbol{u}\left(t^{n+1}\right)\right\Vert _{1+s}^{2}\left\Vert e_{\boldsymbol{u}}^{n}\right\Vert ^{2}+C\tau\left\Vert \boldsymbol{u}(t^{n+1})\right\Vert _{1+s}^{2}\int_{t^{n}}^{t^{n+1}}\left\Vert \boldsymbol{u}_{t}(s)\right\Vert ^{2}ds.\label{eq:erroruJ1}
\end{align}
Similarly, term ${\rm I_{2,2}}$ can be estimated by using (\ref{eq:e_trilinear_form0})-(\ref{eq:e_trilinear_form1})
and Young inequality,
\begin{align}
{\rm {\rm I_{2,2}}} & =\left(\boldsymbol{u}^{n}\cdot\nabla\left(\boldsymbol{u}\left(t^{n+1}\right)-\boldsymbol{u}\left(t^{n}\right)\right),e_{\boldsymbol{u}}^{n+1}\right)-\left(e_{\boldsymbol{u}}^{n}\cdot\nabla e_{\boldsymbol{u}}^{n},e_{\boldsymbol{u}}^{n+1}\right)-\left(\boldsymbol{u}\left(t^{n}\right)\cdot\nabla e_{\boldsymbol{u}}^{n},e_{\boldsymbol{u}}^{n+1}\right)\nonumber \\
 & =\left(\boldsymbol{u}^{n}\cdot\nabla\left(\boldsymbol{u}\left(t^{n+1}\right)-\boldsymbol{u}\left(t^{n}\right)\right),e_{\boldsymbol{u}}^{n+1}\right)+\left(e_{\boldsymbol{u}}^{n}\cdot\nabla e_{\boldsymbol{u}}^{n+1},e_{\boldsymbol{u}}^{n}\right)+\left(\boldsymbol{u}\left(t^{n}\right)\cdot\nabla e_{\boldsymbol{u}}^{n+1},e_{\boldsymbol{u}}^{n}\right)\nonumber \\
 & \le C_{b,1}\left\Vert \boldsymbol{u}^{n}\right\Vert \left\Vert \boldsymbol{u}\left(t^{n+1}\right)-\boldsymbol{u}\left(t^{n}\right)\right\Vert _{1+s}\left\Vert \nabla e_{\boldsymbol{u}}^{n+1}\right\Vert +C_{b,6}\left\Vert e_{\boldsymbol{u}}^{n}\right\Vert \left\Vert \nabla e_{\boldsymbol{u}}^{n}\right\Vert \left\Vert \nabla e_{\boldsymbol{u}}^{n+1}\right\Vert +C_{b,2}\left\Vert \boldsymbol{u}\left(t^{n}\right)\right\Vert _{1+s}\left\Vert \nabla e_{\boldsymbol{u}}^{n+1}\right\Vert \left\Vert e_{\boldsymbol{u}}^{n}\right\Vert \nonumber \\
 & \le C_{b,1}\left\Vert \boldsymbol{u}^{n}\right\Vert \left\Vert \int_{t^{n}}^{t^{n+1}}\boldsymbol{u}_{t}{\rm d}t\right\Vert _{1+s}\left\Vert \nabla e_{\boldsymbol{u}}^{n+1}\right\Vert +C_{b,6}\left\Vert e_{\boldsymbol{u}}^{n}\right\Vert \left\Vert \nabla e_{\boldsymbol{u}}^{n}\right\Vert \left\Vert \nabla e_{\boldsymbol{u}}^{n+1}\right\Vert +C_{b,2}\left\Vert \boldsymbol{u}\left(t^{n}\right)\right\Vert _{1+s}\left\Vert \nabla e_{\boldsymbol{u}}^{n+1}\right\Vert \left\Vert e_{\boldsymbol{u}}^{n}\right\Vert \nonumber \\
 & \leq\frac{R_{e}^{-1}}{8}\left\Vert \nabla e_{\boldsymbol{u}}^{n+1}\right\Vert ^{2}+C\left(\left\Vert \boldsymbol{u}\left(t^{n}\right)\right\Vert _{1+s}^{2}+\left\Vert \nabla e_{\boldsymbol{u}}^{n}\right\Vert ^{2}\right)\left\Vert e_{\boldsymbol{u}}^{n}\right\Vert ^{2}+C\tau\left\Vert \boldsymbol{u}^{n}\right\Vert \int_{t^{n}}^{t^{n+1}}\left\Vert \boldsymbol{u}_{t}(s)\right\Vert _{1+s}^{2}ds.\label{eq:erroruJ2}
\end{align}
For term ${\rm I}_{3}$, we invoke with Hölder inequality and Young
inequality to deduce that 
\begin{align}
{\rm I}_{3} & =\kappa\exp\left(\frac{t^{n+1}}{T}\right)e_{q}^{n+1}\left(\boldsymbol{J}^{n}\times\boldsymbol{B}^{n+1},e_{\boldsymbol{u}}^{n+1}\right)+\left(\left(\boldsymbol{J}^{n}-\boldsymbol{J}\left(t^{n+1}\right)\right)\times\boldsymbol{B}^{n+1},e_{\boldsymbol{u}}^{n+1}\right)\nonumber \\
 & =\kappa\exp\left(\frac{t^{n+1}}{T}\right)e_{q}^{n+1}\left(\boldsymbol{J}^{n}\times\boldsymbol{B}^{n+1},e_{\boldsymbol{u}}^{n+1}\right)+\kappa\left\Vert \boldsymbol{J}^{n}-\boldsymbol{J}\left(t^{n+1}\right)\right\Vert \left\Vert \boldsymbol{B}^{n+1}\right\Vert _{0,3}\left\Vert e_{\boldsymbol{u}}^{n+1}\right\Vert _{0,6}\nonumber \\
 & \le\kappa\exp\left(\frac{t^{n+1}}{T}\right)e_{q}^{n+1}\left(\boldsymbol{J}^{n}\times\boldsymbol{B}^{n+1},e_{\boldsymbol{u}}^{n+1}\right)+C_{p}\kappa\left\Vert e_{\boldsymbol{J}}^{n}+\boldsymbol{J}\left(t^{n}\right)-\boldsymbol{J}\left(t^{n+1}\right)\right\Vert \left\Vert \boldsymbol{B}^{n+1}\right\Vert _{0,3}\left\Vert \nabla e_{\boldsymbol{u}}^{n+1}\right\Vert \nonumber \\
 & \le\kappa\exp\left(\frac{t^{n+1}}{T}\right)e_{q}^{n+1}\left(\boldsymbol{J}^{n}\times\boldsymbol{B}^{n+1},e_{\boldsymbol{u}}^{n+1}\right)+\frac{R_{e}^{-1}}{8}\left\Vert \nabla e_{\boldsymbol{u}}^{n+1}\right\Vert ^{2}\nonumber \\
 & \quad+C\left\Vert \boldsymbol{B}^{n+1}\right\Vert _{0,3}^{2}\left\Vert e_{\boldsymbol{J}}^{n}\right\Vert ^{2}+C\tau\left\Vert \boldsymbol{B}^{n+1}\right\Vert _{0,3}^{2}\int_{t^{n}}^{t^{n+1}}\left\Vert \boldsymbol{J}_{t}\right\Vert ^{2}{\rm d}t.\label{eq:erroruIII}
\end{align}
Combining (\ref{eq:erroruI}) with (\ref{eq:truncu})\textendash (\ref{eq:erroruIII})
leads to the desired result.
\end{proof}
Next, we derive a bound for the errors of the current density.
\begin{lem}
\label{lem:errorJ}Under the assumptions of Theorem \ref{thm:erroruJq},
the following error estimate holds for $\ 0\leq n\leq N-1$,
\begin{equation}
\frac{\kappa}{2}\left\Vert e_{\boldsymbol{J}}^{n+1}\right\Vert ^{2}\le\kappa\exp\left(\frac{t^{n+1}}{T}\right)e_{q}^{n+1}\left(\boldsymbol{u}^{n}\times\boldsymbol{B}^{n+1},e_{\boldsymbol{J}}^{n+1}\right)+C\left\Vert \boldsymbol{B}^{n+1}\right\Vert _{0,3}^{2}\left\Vert \nabla e_{\boldsymbol{u}}^{n}\right\Vert ^{2}+C\tau\left\Vert \boldsymbol{B}^{n+1}\right\Vert _{0,3}^{2}\int_{t^{n}}^{t^{n+1}}\left\Vert \nabla\boldsymbol{u}_{t}\right\Vert ^{2}{\rm d}t.\label{eq:error_J}
\end{equation}
\end{lem}
\begin{proof}
Taking the $\boldsymbol{L}^{2}$-inner product of (\ref{eq:errorJ})
with $\kappa e_{\boldsymbol{J}}^{n+1}$ and utilizing Hölder inequality
and Young inequality, we obtain
\begin{align}
\kappa\left\Vert e_{\boldsymbol{J}}^{n+1}\right\Vert ^{2} & =\kappa\exp\left(\frac{t^{n+1}}{T}\right)\left(\left(q^{n+1}\boldsymbol{u}^{n}\times\boldsymbol{B}^{n+1}-q\left(t^{n+1}\right)\boldsymbol{u}\left(t^{n+1}\right)\times\boldsymbol{B}^{n+1}\right),e_{\boldsymbol{J}}^{n+1}\right)\nonumber \\
 & =\kappa\exp\left(\frac{t^{n+1}}{T}\right)e_{q}^{n+1}\left(\boldsymbol{u}^{n}\times\boldsymbol{B}^{n+1},e_{\boldsymbol{J}}^{n+1}\right)-\kappa\left(\left(\boldsymbol{u}\left(t^{n+1}\right)-\boldsymbol{u}^{n}\right)\times\boldsymbol{B}^{n+1},e_{\boldsymbol{J}}^{n+1}\right)\nonumber \\
 & \le\kappa\exp\left(\frac{t^{n+1}}{T}\right)e_{q}^{n+1}\left(\boldsymbol{u}^{n}\times\boldsymbol{B}^{n+1},e_{\boldsymbol{J}}^{n+1}\right)+C_{p}\kappa\left\Vert \nabla\left(\boldsymbol{u}\left(t^{n+1}\right)-\boldsymbol{u}^{n}\right)\right\Vert \left\Vert e_{\boldsymbol{J}}^{n+1}\right\Vert \left\Vert \boldsymbol{B}^{n+1}\right\Vert _{0,3}\nonumber \\
 & \le\kappa\exp\left(\frac{t^{n+1}}{T}\right)e_{q}^{n+1}\left(\boldsymbol{u}^{n}\times\boldsymbol{B}^{n+1},e_{\boldsymbol{J}}^{n+1}\right)+\frac{\kappa}{2}\left\Vert e_{\boldsymbol{J}}^{n+1}\right\Vert ^{2}\nonumber \\
 & \quad+C\left\Vert \boldsymbol{B}^{n+1}\right\Vert _{0,3}^{2}\left\Vert \nabla e_{\boldsymbol{u}}^{n}\right\Vert ^{2}+C\tau\left\Vert \boldsymbol{B}^{n+1}\right\Vert _{0,3}^{2}\int_{t^{n}}^{t^{n+1}}\left\Vert \nabla\boldsymbol{u}_{t}\right\Vert ^{2}{\rm d}t.\label{eq:errorJI}
\end{align}
This leads to the desired result.
\end{proof}
In the next lemma, we derive a bound for the errors with respect to
$q$.
\begin{lem}
\label{lem:errorq}Under the assumptions of Theorem \ref{thm:erroruJq},
the following error estimate holds for $\ 0\leq n\leq N-1$,
\begin{align}
 & \frac{|e_{q}^{n+1}|^{2}-|e_{q}^{n}|^{2}+|e_{q}^{n+1}-e_{q}^{n}|^{2}}{2\tau}+\frac{1}{2T}|e_{q}^{n+1}|^{2}\nonumber \\
 & =\exp\left(\frac{t^{n+1}}{T}\right)e_{q}^{n+1}\left(\boldsymbol{u}^{n}\cdot\nabla\boldsymbol{u}^{n},e_{\boldsymbol{u}}^{n+1}\right)-\kappa\exp\left(\frac{t^{n+1}}{T}\right)e_{q}^{n+1}\left(\boldsymbol{J}^{n}\times\boldsymbol{B}^{n+1},e_{\boldsymbol{u}}^{n+1}\right)\nonumber \\
 & \quad-\kappa\exp\left(\frac{t^{n+1}}{T}\right)e_{q}^{n+1}\left(\boldsymbol{u}^{n}\times\boldsymbol{B}^{n+1},e_{\boldsymbol{J}}^{n+1}\right)+\frac{1}{4k_{2}}\left\Vert \nabla\boldsymbol{u}^{n}\right\Vert ^{2}|e_{q}^{n+1}|^{2}\nonumber \\
 & \quad+C\left(\left\Vert \boldsymbol{u}(t^{n+1})\right\Vert _{1+s}^{2}+\left\Vert \nabla\boldsymbol{u}(t^{n+1})\right\Vert ^{2}\left\Vert \boldsymbol{u}(t^{n+1})\right\Vert _{1+s}^{2}+\left\Vert \boldsymbol{B}^{n+1}\right\Vert _{0,3}^{2}\left\Vert \boldsymbol{J}\left(t^{n+1}\right)\right\Vert ^{2}\right)\left\Vert e_{\boldsymbol{u}}^{n}\right\Vert ^{2}\nonumber \\
 & \quad+C\left\Vert \boldsymbol{B}^{n+1}\right\Vert _{0,3}^{2}\left\Vert \nabla\boldsymbol{u}\left(t^{n+1}\right)\right\Vert ^{2}\left\Vert e_{\boldsymbol{J}}^{n}\right\Vert ^{2}+C\tau\left\Vert \boldsymbol{B}^{n+1}\right\Vert _{0,3}^{2}\left\Vert \nabla\boldsymbol{u}\left(t^{n+1}\right)\right\Vert ^{2}\int_{t^{n}}^{t^{n+1}}\left\Vert \boldsymbol{J}_{t}\right\Vert ^{2}{\rm d}t+C\tau\int_{t^{n}}^{t^{n+1}}\left|q_{tt}(s)\right|^{2}ds\nonumber \\
 & \quad+C\tau\left(\left\Vert \boldsymbol{u}(t^{n+1})\right\Vert _{1+s}^{2}+\left\Vert \nabla\boldsymbol{u}(t^{n+1})\right\Vert ^{2}\left\Vert \boldsymbol{u}\left(t^{n+1}\right)\right\Vert _{1+s}^{2}+\left\Vert \boldsymbol{B}^{n+1}\right\Vert _{0,3}^{2}\left\Vert \boldsymbol{J}\left(t^{n+1}\right)\right\Vert ^{2}\right)\int_{t^{n}}^{t^{n+1}}\left\Vert \nabla\boldsymbol{u}_{t}(s)\right\Vert ^{2}ds.\label{eq:error_q}
\end{align}
\end{lem}
\begin{proof}
Multiplying both sides of (\ref{eq:errorq}) by $e_{q}^{n+1}$ yields
\begin{align}
 & \frac{\left|e_{q}^{n+1}\right|^{2}-\left|e_{q}^{n}\right|^{2}+\left|e_{q}^{n+1}-e_{q}^{n}\right|^{2}}{2\tau}+\frac{1}{T}\left|e_{q}^{n+1}\right|^{2}\nonumber \\
 & =R_{q}^{n+1}e_{q}^{n+1}+\exp\left(\frac{t^{n+1}}{T}\right)e_{q}^{n+1}\left(\left(\boldsymbol{u}^{n}\cdot\nabla\boldsymbol{u}^{n},\boldsymbol{u}^{n+1}\right)-\left(\boldsymbol{u}\left(t^{n+1}\right)\cdot\nabla\boldsymbol{u}\left(t^{n+1}\right),\boldsymbol{u}\left(t^{n+1}\right)\right)\right)\nonumber \\
 & \quad-\kappa\exp\left(\frac{t^{n+1}}{T}\right)e_{q}^{n+1}\left(\left(\boldsymbol{J}^{n}\times\boldsymbol{B}^{n+1},\boldsymbol{u}^{n+1}\right)-\left(\boldsymbol{J}\left(t^{n+1}\right)\times\boldsymbol{B}^{n+1},\boldsymbol{u}\left(t^{n+1}\right)\right)\right)\nonumber \\
 & \quad-\kappa\exp\left(\frac{t^{n+1}}{T}\right)e_{q}^{n+1}\left(\left(\boldsymbol{u}^{n}\times\boldsymbol{B}^{n+1},\boldsymbol{J}^{n+1}\right)-\left(\boldsymbol{u}\left(t^{n+1}\right)\times\boldsymbol{B}^{n+1},\boldsymbol{J}\left(t^{n+1}\right)\right)\right)\nonumber \\
 & \coloneqq\sum_{i=1}^{4}{\rm I}_{i}\label{eq:errorqI}
\end{align}
We bound term ${\rm I}_{1}$ by using the Young inequality, 
\begin{equation}
{\rm I}_{1}\leq\frac{1}{8T}\left|e_{q}^{n+1}\right|^{2}+C\tau\int_{t^{n}}^{t^{n+1}}\left|q_{tt}(s)\right|^{2}ds.\label{eq:truncq}
\end{equation}

Term ${\rm I}_{2}$ can be recast as 
\begin{align}
{\rm I}_{2}= & \exp\left(\frac{t^{n+1}}{T}\right)e_{q}^{n+1}\left(\boldsymbol{u}^{n}\cdot\nabla\boldsymbol{u}^{n},e_{\boldsymbol{u}}^{n+1}\right)+\exp\left(\frac{t^{n+1}}{T}\right)e_{q}^{n+1}\left(\boldsymbol{u}^{n}\cdot\nabla\left(\boldsymbol{u}^{n}-\boldsymbol{u}\left(t^{n+1}\right)\right),\boldsymbol{u}\left(t^{n+1}\right)\right)\nonumber \\
 & +\exp\left(\frac{t^{n+1}}{T}\right)e_{q}^{n+1}\left(\left(\boldsymbol{u}^{n}-\boldsymbol{u}\left(t^{n+1}\right)\right)\cdot\nabla\boldsymbol{u}\left(t^{n+1}\right),\boldsymbol{u}\left(t^{n+1}\right)\right)\nonumber \\
\coloneqq & \exp\left(\frac{t^{n+1}}{T}\right)e_{q}^{n+1}\left(\boldsymbol{u}^{n}\cdot\nabla\boldsymbol{u}^{n},e_{\boldsymbol{u}}^{n+1}\right)+{\rm I}_{2,1}+{\rm I}_{2,2}.\label{eq:errorqII}
\end{align}
Using (\ref{eq:e_trilinear_form3}) and (\ref{eq:gradu_uniform}),
we bound term ${\rm I}_{2,1}$ by 
\begin{align}
{\rm I}_{2,1} & =\exp\left(\frac{t^{n+1}}{T}\right)e_{q}^{n+1}\left(\boldsymbol{u}^{n}\cdot\nabla\boldsymbol{u}\left(t^{n+1}\right),\boldsymbol{u}^{n}-\boldsymbol{u}\left(t^{n+1}\right)\right)\nonumber \\
 & \le\exp(1)C_{b,3}\left\Vert \nabla\boldsymbol{u}^{n}\right\Vert \left\Vert \boldsymbol{u}\left(t^{n+1}\right)\right\Vert _{1+s,2}\left\Vert \boldsymbol{u}^{n}-\boldsymbol{u}\left(t^{n+1}\right)\right\Vert _{0}\left|e_{q}^{n+1}\right|\nonumber \\
 & \le C\left\Vert \nabla\boldsymbol{u}^{n}\right\Vert \left\Vert \boldsymbol{u}\left(t^{n+1}\right)\right\Vert _{1+s,2}\left\Vert \boldsymbol{u}\left(t^{n+1}\right)-\boldsymbol{u}\left(t^{n}\right)-e_{\boldsymbol{u}}^{n}\right\Vert _{0}\left|e_{q}^{n+1}\right|\nonumber \\
 & \leq\frac{1}{4k_{2}}\left\Vert \nabla\boldsymbol{u}^{n}\right\Vert ^{2}\left|e_{q}^{n+1}\right|^{2}+C\left\Vert e_{\boldsymbol{u}}^{n}\right\Vert ^{2}\left\Vert \boldsymbol{u}\left(t^{n+1}\right)\right\Vert _{1+s,2}^{2}+C\tau\left\Vert \boldsymbol{u}\left(t^{n+1}\right)\right\Vert _{1+s,2}^{2}\int_{t^{n}}^{t^{n+1}}\left\Vert \boldsymbol{u}_{t}\right\Vert ^{2}{\rm d}t,\label{eq:errorqII1}
\end{align}
where $k_{2}$ is given by (\ref{eq:gradu_uniform}). In a same manner,
term ${\rm I}_{2,2}$ can be bounded by using (\ref{eq:e_trilinear_form4})
and Young inequality,
\begin{align}
{\rm I}_{2,2} & =\exp\left(\frac{t^{n+1}}{T}\right)e_{q}^{n+1}\left(\left(\boldsymbol{u}^{n}-\boldsymbol{u}\left(t^{n+1}\right)\right)\cdot\nabla\boldsymbol{u}\left(t^{n+1}\right),\boldsymbol{u}\left(t^{n+1}\right)\right)\nonumber \\
 & \leq\exp(1)C_{b,4}\left|e_{q}^{n+1}\right|\left\Vert \boldsymbol{u}\left(t^{n+1}\right)-\boldsymbol{u}^{n}\right\Vert \left\Vert \nabla\boldsymbol{u}\left(t^{n+1}\right)\right\Vert \left\Vert \boldsymbol{u}\left(t^{n+1}\right)\right\Vert _{1+s,2}\nonumber \\
 & \le C\left|e_{q}^{n+1}\right|\left\Vert \boldsymbol{u}\left(t^{n+1}\right)-\boldsymbol{u}\left(t^{n}\right)-e_{\boldsymbol{u}}^{n}\right\Vert \left\Vert \nabla\boldsymbol{u}\left(t^{n+1}\right)\right\Vert \left\Vert \boldsymbol{u}\left(t^{n+1}\right)\right\Vert _{1+s,2}\nonumber \\
 & \leq\frac{1}{8T}\left|e_{q}^{n+1}\right|^{2}+C\left\Vert \nabla\boldsymbol{u}\left(t^{n+1}\right)\right\Vert ^{2}\left\Vert \boldsymbol{u}\left(t^{n+1}\right)\right\Vert _{2}^{2}\left\Vert e_{\boldsymbol{u}}^{n}\right\Vert ^{2}+C\tau\left\Vert \nabla\boldsymbol{u}\left(t^{n+1}\right)\right\Vert ^{2}\left\Vert \boldsymbol{u}\left(t^{n+1}\right)\right\Vert _{1+s}^{2}\int_{t^{n}}^{t^{n+1}}\left\Vert \boldsymbol{u}_{t}\right\Vert ^{2}{\rm d}t.\label{eq:errorqII2}
\end{align}

Using Hölder inequality and Young inequality, term ${\rm I}_{3}$
can be estimated by 
\begin{align}
{\rm I}_{3} & =\kappa\exp\left(\frac{t^{n+1}}{T}\right)e_{q}^{n+1}\left(\left(\boldsymbol{J}\left(t^{n+1}\right)-\boldsymbol{J}^{n}\right)\times\boldsymbol{B}^{n+1},\boldsymbol{u}\left(t^{n+1}\right)\right)-\kappa\exp\left(\frac{t^{n+1}}{T}\right)e_{q}^{n+1}\left(\boldsymbol{J}^{n}\times\boldsymbol{B}^{n+1},e_{\boldsymbol{u}}^{n+1}\right)\nonumber \\
 & \le\kappa\exp(1)\left|e_{q}^{n+1}\right|\left\Vert \boldsymbol{J}\left(t^{n+1}\right)-\boldsymbol{J}^{n}\right\Vert \left\Vert \boldsymbol{B}^{n+1}\right\Vert _{0,3}\left\Vert \boldsymbol{u}\left(t^{n+1}\right)\right\Vert _{0,6}-\kappa\exp\left(\frac{t^{n+1}}{T}\right)e_{q}^{n+1}\left(\boldsymbol{J}^{n}\times\boldsymbol{B}^{n+1},e_{\boldsymbol{u}}^{n+1}\right)\nonumber \\
 & \le C\left|e_{q}^{n+1}\right|\left\Vert \boldsymbol{J}\left(t^{n+1}\right)-\boldsymbol{J}\left(t^{n}\right)-e_{\boldsymbol{J}}^{n}\right\Vert \left\Vert \boldsymbol{B}^{n+1}\right\Vert _{0,3}\left\Vert \nabla\boldsymbol{u}\left(t^{n+1}\right)\right\Vert -\kappa\exp\left(\frac{t^{n+1}}{T}\right)e_{q}^{n+1}\left(\boldsymbol{J}^{n}\times\boldsymbol{B}^{n+1},e_{\boldsymbol{u}}^{n+1}\right)\nonumber \\
 & \le\frac{1}{8T}\left|e_{q}^{n+1}\right|^{2}+C\left\Vert \boldsymbol{B}^{n+1}\right\Vert _{0,3}^{2}\left\Vert \nabla\boldsymbol{u}\left(t^{n+1}\right)\right\Vert ^{2}\left\Vert e_{\boldsymbol{J}}^{n}\right\Vert ^{2}+C\tau\left\Vert \boldsymbol{B}^{n+1}\right\Vert _{0,3}^{2}\left\Vert \nabla\boldsymbol{u}\left(t^{n+1}\right)\right\Vert ^{2}\int_{t^{n}}^{t^{n+1}}\left\Vert \boldsymbol{J}_{t}\right\Vert ^{2}{\rm d}t\nonumber \\
 & \quad-\kappa\exp\left(\frac{t^{n+1}}{T}\right)e_{q}^{n+1}\left(\boldsymbol{J}^{n}\times\boldsymbol{B}^{n+1},e_{\boldsymbol{u}}^{n+1}\right).\label{eq:errorqIII}
\end{align}

Using the similar procedure, term ${\rm I}_{4}$ can be bounded by
\begin{align}
{\rm I}_{4} & =-\kappa\exp\left(\frac{t^{n+1}}{T}\right)e_{q}^{n+1}\left(\boldsymbol{B}^{n+1}\times\boldsymbol{u}^{n},e_{\boldsymbol{J}}^{n+1}\right)-\kappa\exp\left(\frac{t^{n+1}}{T}\right)e_{q}^{n+1}\left(\boldsymbol{B}^{n+1}\times\left(\boldsymbol{u}^{n}-\boldsymbol{u}\left(t^{n+1}\right)\right),\boldsymbol{J}\left(t^{n+1}\right)\right)\nonumber \\
 & \le-\kappa\exp\left(\frac{t^{n+1}}{T}\right)e_{q}^{n+1}\left(\boldsymbol{B}^{n+1}\times\boldsymbol{u}^{n},e_{\boldsymbol{J}}^{n+1}\right)+\kappa\exp(1)\left|e_{q}^{n+1}\right|\left\Vert \boldsymbol{u}^{n}-\boldsymbol{u}\left(t^{n+1}\right)\right\Vert _{0,6}\left\Vert \boldsymbol{B}^{n+1}\right\Vert _{0,3}\left\Vert \boldsymbol{J}\left(t^{n+1}\right)\right\Vert \nonumber \\
 & \le-\kappa\exp\left(\frac{t^{n+1}}{T}\right)e_{q}^{n+1}\left(\boldsymbol{B}^{n+1}\times\boldsymbol{u}^{n},e_{\boldsymbol{J}}^{n+1}\right)+C\left|e_{q}^{n+1}\right|\left\Vert \nabla\left(e_{\boldsymbol{u}}^{n}+\boldsymbol{u}\left(t^{n}\right)-\boldsymbol{u}\left(t^{n+1}\right)\right)\right\Vert \left\Vert \boldsymbol{B}^{n+1}\right\Vert _{0,3}\left\Vert \boldsymbol{J}\left(t^{n+1}\right)\right\Vert \nonumber \\
 & \le-\kappa\exp\left(\frac{t^{n+1}}{T}\right)e_{q}^{n+1}\left(\boldsymbol{B}^{n+1}\times\boldsymbol{u}^{n},e_{\boldsymbol{J}}^{n+1}\right)+\frac{1}{8T}\left|e_{q}^{n+1}\right|^{2}+C\left\Vert \boldsymbol{B}^{n+1}\right\Vert _{0,3}^{2}\left\Vert \boldsymbol{J}\left(t^{n+1}\right)\right\Vert ^{2}\left\Vert \nabla e_{\boldsymbol{u}}^{n}\right\Vert ^{2}\nonumber \\
 & \quad+C\tau\left\Vert \boldsymbol{B}^{n+1}\right\Vert _{0,3}^{2}\left\Vert \boldsymbol{J}\left(t^{n+1}\right)\right\Vert ^{2}\int_{t^{n}}^{t^{n+1}}\left\Vert \nabla\boldsymbol{u}_{t}\right\Vert ^{2}{\rm d}t.\label{eq:errorqIV}
\end{align}
Combining (\ref{eq:errorqI}) with (\ref{eq:truncq})-(\ref{eq:errorqIV})
leads to the desired result.
\end{proof}
\begin{rem}
Stability estimates for the discrete solutions play a key role in the error analysis. Such type of error estimates for MHD model can been found in \cite{Li2021}. However, the
proof therein seems to be not complete and there are some minor typos in the similar estimate of \eqref{eq:errorqII1}. Here we correct it and give a rigorous proof.
\end{rem}
\medskip{}

Now we are in the position to prove Theorem \ref{thm:erroruJq} by
using Lemmas \ref{lem:erroru}-\ref{lem:errorq}.
\begin{proof}[Proof of Theorem \ref{thm:erroruJq}.]
 Summing up (\ref{eq:error_u}), (\ref{eq:error_J}) and (\ref{eq:error_q})
leads to
\begin{equation}
\begin{aligned} & \frac{\left\Vert e_{\boldsymbol{u}}^{n+1}\right\Vert ^{2}-\left\Vert e_{\boldsymbol{u}}^{n}\right\Vert ^{2}+\left\Vert e_{\boldsymbol{u}}^{n+1}-e_{\boldsymbol{u}}^{n}\right\Vert ^{2}}{2\tau}+\frac{R_{e}^{-1}}{2}\left\Vert \nabla e_{\boldsymbol{u}}^{n+1}\right\Vert ^{2}+\frac{\kappa}{2}\left\Vert e_{\boldsymbol{J}}^{n+1}\right\Vert ^{2}\\
 & +\frac{\left|e_{q}^{n+1}\right|^{2}-\left|e_{q}^{n}\right|^{2}+\left|e_{q}^{n+1}-e_{q}^{n}\right|^{2}}{2\tau}+\frac{1}{2T}\left|e_{q}^{n+1}\right|^{2}\\
 & \le\frac{1}{4k_{2}}\left\Vert \nabla\boldsymbol{u}^{n}\right\Vert ^{2}|e_{q}^{n+1}|^{2}+C\left(1+\left\Vert \nabla e_{\boldsymbol{u}}^{n}\right\Vert ^{2}\right)\left\Vert e_{\boldsymbol{u}}^{n}\right\Vert ^{2}+C\left\Vert \nabla e_{\boldsymbol{u}}^{n}\right\Vert ^{2}+C\left\Vert e_{\boldsymbol{J}}^{n}\right\Vert ^{2}\\
 & \quad+C\tau\int_{t^{n}}^{t^{n+1}}\left(\left\Vert \boldsymbol{u}_{t}\right\Vert _{1+s}^{2}+\left\Vert \boldsymbol{u}_{tt}\right\Vert _{-1}^{2}+\left|q_{tt}\right|^{2}+\left\Vert \boldsymbol{J}_{t}\right\Vert ^{2}\right){\rm d}t.
\end{aligned}
\label{eq:errorthm}
\end{equation}
We will prove the error estimate by mathematical induction argument.
For the case of $k=m=0$, by using (\ref{eq:errorthm}), we have 
\begin{align*}
 & \frac{\left\Vert e_{\boldsymbol{u}}^{1}\right\Vert ^{2}-\left\Vert e_{\boldsymbol{u}}^{0}\right\Vert ^{2}+\left\Vert e_{\boldsymbol{u}}^{1}-e_{\boldsymbol{u}}^{0}\right\Vert ^{2}}{2\tau}+\frac{\left|e_{q}^{1}\right|^{2}-\left|e_{q}^{0}\right|^{2}+\left|e_{q}^{1}-e_{q}^{0}\right|^{2}}{2\tau}\\
 & \quad+\frac{R_{e}^{-1}}{2}\left\Vert \nabla e_{\boldsymbol{u}}^{1}\right\Vert ^{2}+\tau\kappa\left\Vert e_{\boldsymbol{J}}^{1}\right\Vert ^{2}+\frac{1}{2T}\left|e_{q}^{1}\right|^{2}\\
 & \le\frac{1}{4k_{2}}\left\Vert \nabla\boldsymbol{u}^{0}\right\Vert ^{2}|e_{q}^{1}|^{2}+C\left\Vert e_{\boldsymbol{u}}^{0}\right\Vert ^{2}+C\left\Vert \nabla e_{\boldsymbol{u}}^{0}\right\Vert ^{2}+C\left\Vert e_{\boldsymbol{J}}^{0}\right\Vert ^{2}+C\tau.
\end{align*}
Invoking with $e_{\boldsymbol{u}}^{0}=e_{\boldsymbol{J}}^{0}=\boldsymbol{0}$
and (\ref{eq:gradu_uniform}), we can easily get
\begin{align*}
 & \left\Vert e_{\boldsymbol{u}}^{1}\right\Vert ^{2}+\left|e_{q}^{1}\right|^{2}+\tau R_{e}^{-1}\left\Vert \nabla e_{\boldsymbol{u}}^{1}\right\Vert ^{2}+\tau\kappa\left\Vert e_{\boldsymbol{J}}^{1}\right\Vert ^{2}\\
 & \qquad+\left\Vert e_{\boldsymbol{u}}^{1}-e_{\boldsymbol{u}}^{0}\right\Vert ^{2}+\left|e_{q}^{1}-e_{q}^{0}\right|^{2}\leq C\tau^{2}.
\end{align*}
This means that the error estimate (\ref{eq:estuJq}) holds for $m=0$.
For $0<k\le m-1$, we assume that the error estimate (\ref{eq:estuJq})
is valid. We will show it is also valid for $k=m$. We first deduce
a bound for $\left|e_{q}^{m^{*}+1}\right|$, where $m^{*}$ is the
time step such that 
\begin{equation}
\left|e_{q}^{m^{*}+1}\right|=\max_{0\leq n\leq m}\left|e_{q}^{n+1}\right|.\label{eq:eqmax}
\end{equation}
Multiplying (\ref{eq:errorthm}) by $2\tau$, summing up over $n$
from 0 to $m^{*}$, using (\ref{eq:eqmax}) and the recursive hypothesis,
we get
\begin{align}
 & \left\Vert e_{\boldsymbol{u}}^{m^{*}+1}\right\Vert ^{2}+\tau\sum_{n=0}^{m^{*}}\left(R_{e}^{-1}\left\Vert \nabla e_{\boldsymbol{u}}^{n+1}\right\Vert ^{2}+\kappa\left\Vert e_{\boldsymbol{J}}^{n+1}\right\Vert ^{2}\right)+|e_{q}^{m^{*}+1}|^{2}+\frac{\tau}{T}\sum_{n=0}^{m^{*}}|e_{q}^{n+1}|^{2}\nonumber \\
 & \quad+\stackrel[n=0]{m^{*}}{\sum}\left\Vert e_{\boldsymbol{u}}^{n+1}-e_{\boldsymbol{u}}^{n}\right\Vert ^{2}+\stackrel[n=0]{m^{*}}{\sum}\left|e_{q}^{n+1}-e_{q}^{n}\right|^{2}\nonumber \\
 & \le\frac{\tau}{2k_{2}}\sum_{n=0}^{m^{*}}\left\Vert \nabla\boldsymbol{u}^{n}\right\Vert ^{2}|e_{q}^{m^{*}+1}|^{2}+C\tau\sum_{n=0}^{m^{*}}\left(1+\left\Vert \nabla e_{\boldsymbol{u}}^{n}\right\Vert ^{2}\right)\left\Vert e_{\boldsymbol{u}}^{n}\right\Vert ^{2}+C\tau\sum_{n=0}^{m^{*}}\left\Vert \nabla e_{\boldsymbol{u}}^{n}\right\Vert ^{2}\nonumber \\
 & \quad+C\tau\stackrel[n=0]{m^{*}}{\sum}\left\Vert e_{\boldsymbol{J}}^{n}\right\Vert ^{2}+C\tau\int_{0}^{t^{m^{*}+1}}\left(\left\Vert \boldsymbol{u}_{t}\right\Vert _{2}^{2}+\left\Vert \boldsymbol{u}_{tt}\right\Vert _{-1}^{2}+\left\Vert q_{tt}\right\Vert ^{2}\right){\rm d}t,\nonumber \\
 & \le\frac{\tau}{2k_{2}}\sum_{n=0}^{m^{*}}\left\Vert \nabla\boldsymbol{u}^{n}\right\Vert ^{2}|e_{q}^{m^{*}+1}|^{2}+C\tau\sum_{n=0}^{m^{*}}\left(1+\left\Vert \nabla e_{\boldsymbol{u}}^{n}\right\Vert ^{2}\right)\left\Vert e_{\boldsymbol{u}}^{n}\right\Vert ^{2}+C\tau^{2}\nonumber \\
 & \quad+C\tau\int_{0}^{t^{m^{*}+1}}\left(\left\Vert \boldsymbol{u}_{t}\right\Vert _{2}^{2}+\left\Vert \boldsymbol{u}_{tt}\right\Vert _{-1}^{2}+\left\Vert q_{tt}\right\Vert ^{2}\right){\rm d}t.\label{eq:errorthm1}
\end{align}
It follows from (\ref{eq:gradu_uniform}) that
\[
\frac{\tau}{2k_{2}}\sum_{n=0}^{m^{*}}\left\Vert \nabla\boldsymbol{u}^{n}\right\Vert ^{2}|e_{q}^{m^{*}+1}|^{2}\le\frac{1}{2}\left|e_{q}^{m^{*}+1}\right|^{2},\quad\tau\sum_{n=0}^{m^{*}}\left(1+\left\Vert \nabla e_{\boldsymbol{u}}^{n}\right\Vert ^{2}\right)\le C.
\]
Invoking with the discrete Gronwall inequality in Lemma \ref{lem: gronwall2},
we obtain
\begin{align}
 & \left\Vert e_{\boldsymbol{u}}^{m^{*}+1}\right\Vert ^{2}+\tau\sum_{n=0}^{m^{*}}\left(R_{e}^{-1}\left\Vert \nabla e_{\boldsymbol{u}}^{n+1}\right\Vert ^{2}+\kappa\left\Vert e_{\boldsymbol{J}}^{n+1}\right\Vert ^{2}\right)+|e_{q}^{m^{*}+1}|^{2}+\frac{\tau}{T}\sum_{n=0}^{m^{*}}|e_{q}^{n+1}|^{2}\nonumber \\
 & \le C\tau^{2}\int_{0}^{t^{m^{*}+1}}\left(\left\Vert \boldsymbol{u}_{t}\right\Vert _{2}^{2}+\left\Vert \boldsymbol{u}_{tt}\right\Vert _{-1}^{2}+\left\Vert q_{tt}\right\Vert ^{2}\right){\rm d}t.\label{eq:errorthm2}
\end{align}
Now we turn to (\ref{eq:errorthm}), multiply it by $2\tau$ and sum
up over $n$ from 0 to $m$, and using (\ref{eq:eqmax}),
\begin{align}
 & \left\Vert e_{\boldsymbol{u}}^{m+1}\right\Vert ^{2}+\left|e_{q}^{m+1}\right|^{2}+\tau R_{e}^{-1}\stackrel[n=0]{m}{\sum}\left\Vert \nabla e_{\boldsymbol{u}}^{n+1}\right\Vert ^{2}+\tau\kappa\stackrel[n=0]{m}{\sum}\left\Vert e_{\boldsymbol{J}}^{n+1}\right\Vert ^{2}\nonumber \\
 & \quad+\stackrel[n=0]{m}{\sum}\left\Vert e_{\boldsymbol{u}}^{n+1}-e_{\boldsymbol{u}}^{n}\right\Vert ^{2}+\stackrel[n=0]{m}{\sum}\left|e_{q}^{n+1}-e_{q}^{n}\right|^{2}\nonumber \\
 & \le\frac{\tau}{2k_{2}}\sum_{n=0}^{m}\left\Vert \nabla\boldsymbol{u}^{n}\right\Vert ^{2}|e_{q}^{m+1}|^{2}+C\tau\sum_{n=0}^{m}\left(1+\left\Vert \nabla e_{\boldsymbol{u}}^{n}\right\Vert ^{2}\right)\left\Vert e_{\boldsymbol{u}}^{n}\right\Vert ^{2}+C\tau\sum_{n=0}^{m}\left\Vert \nabla e_{\boldsymbol{u}}^{n}\right\Vert ^{2}\nonumber \\
 & \quad+C\tau\stackrel[n=0]{m^{*}}{\sum}\left\Vert e_{\boldsymbol{J}}^{n}\right\Vert ^{2}+C\tau\int_{0}^{t^{m+1}}\left(\left\Vert \boldsymbol{u}_{t}\right\Vert _{2}^{2}+\left\Vert \boldsymbol{u}_{tt}\right\Vert _{-1}^{2}+\left\Vert q_{tt}\right\Vert ^{2}\right){\rm d}t.\nonumber \\
 & \le\frac{\tau}{2k_{2}}\sum_{n=0}^{m}\left\Vert \nabla\boldsymbol{u}^{n}\right\Vert ^{2}|e_{q}^{m+1}|^{2}+C\tau\sum_{n=0}^{m}\left(1+\left\Vert \nabla e_{\boldsymbol{u}}^{n}\right\Vert ^{2}\right)\left\Vert e_{\boldsymbol{u}}^{n}\right\Vert ^{2}+C\tau^{2}\nonumber \\
 & \quad+C\tau\int_{0}^{t^{m+1}}\left(\left\Vert \boldsymbol{u}_{t}\right\Vert _{2}^{2}+\left\Vert \boldsymbol{u}_{tt}\right\Vert _{-1}^{2}+\left\Vert q_{tt}\right\Vert ^{2}\right){\rm d}t.\label{eq:errorthm3}
\end{align}
From (\ref{eq:gradu_uniform}) again, we have that
\[
\tau\sum_{n=0}^{m}\left(1+\left\Vert \nabla e_{\boldsymbol{u}}^{n}\right\Vert ^{2}\right)\le C.
\]
Using (\ref{eq:errorthm2}) and the discrete Gronwall inequality in
Lemma \ref{lem: gronwall2} to (\ref{eq:errorthm3}), we complete
the proof.
\end{proof}

\subsection{Error estimates for the pressure and electric potential}

This section is devoted to presenting error estimates for the pressure
and electric potential. We first deduce the error estimates for the
electric potential.
\begin{thm}
Under the assumption of Theorem \ref{thm:erroruJq}, the following
error estimate holds for $\ 0\leq m\leq N-1$,
\begin{equation}
\tau\sum_{n=0}^{m}\left\Vert e_{\phi}^{n+1}\right\Vert ^{2}\leq C\tau^{2}.\label{eq:error_phi}
\end{equation}
\end{thm}
\begin{proof}
\noindent By taking the inner product of (\ref{eq:errorJ}) with $\boldsymbol{K}\in\boldsymbol{D}$
and using the similar arguments in (\ref{eq:errorJI}), we obtain
\begin{align*}
\left(e_{\phi}^{n+1},{\rm div}\boldsymbol{K}\right) & =\left(e_{\boldsymbol{J}}^{n+1},\boldsymbol{K}\right)-\exp\left(\frac{t^{n+1}}{T}\right)\left(q^{n+1}\boldsymbol{u}^{n}\times\boldsymbol{B}^{n+1}-q\left(t^{n+1}\right)\boldsymbol{u}\left(t^{n+1}\right)\times\boldsymbol{B}^{n+1},\boldsymbol{K}\right)\\
 & \le\left\Vert e_{\boldsymbol{J}}^{n+1}\right\Vert \left\Vert \boldsymbol{K}\right\Vert +\exp(1)\left|e_{q}^{n+1}\right|\left\Vert \boldsymbol{u}^{n}\right\Vert _{0,6}\left\Vert \boldsymbol{B}^{n+1}\right\Vert _{0,3}\left\Vert \boldsymbol{K}\right\Vert +\left\Vert \boldsymbol{u}\left(t^{n+1}\right)-\boldsymbol{u}^{n}\right\Vert _{0,6}\left\Vert \boldsymbol{B}^{n+1}\right\Vert _{0,3}\left\Vert \boldsymbol{K}\right\Vert \\
 & \le\left\Vert e_{\boldsymbol{J}}^{n+1}\right\Vert \left\Vert \boldsymbol{K}\right\Vert _{{\rm div}}+C\left|e_{q}^{n+1}\right|\left\Vert \nabla\boldsymbol{u}^{n}\right\Vert \left\Vert \boldsymbol{B}^{n+1}\right\Vert _{0,3}\left\Vert \boldsymbol{K}\right\Vert _{{\rm div}}+C\left\Vert \nabla\left(\boldsymbol{u}\left(t^{n+1}\right)-\boldsymbol{u}^{n}\right)\right\Vert \left\Vert \boldsymbol{B}^{n+1}\right\Vert _{0,3}\left\Vert \boldsymbol{K}\right\Vert _{{\rm div}}\\
 & \le\left\Vert e_{\boldsymbol{J}}^{n+1}\right\Vert \left\Vert \boldsymbol{K}\right\Vert _{{\rm div}}+C\left|e_{q}^{n+1}\right|\left\Vert \nabla\boldsymbol{u}^{n}\right\Vert \left\Vert \boldsymbol{B}^{n+1}\right\Vert _{0,3}\left\Vert \boldsymbol{K}\right\Vert _{{\rm div}}+C\left\Vert \nabla e_{\boldsymbol{u}}^{n}\right\Vert \left\Vert \boldsymbol{B}^{n+1}\right\Vert _{0,3}\left\Vert \boldsymbol{K}\right\Vert _{{\rm div}}\\
 & \quad+C\tau^{\frac{1}{2}}\left(\int_{t^{n}}^{t^{n+1}}\left\Vert \nabla\boldsymbol{u}_{t}\right\Vert ^{2}{\rm d}t\right)^{\frac{1}{2}}\left\Vert \boldsymbol{B}^{n+1}\right\Vert _{0,3}\left\Vert \boldsymbol{K}\right\Vert _{{\rm div}}.
\end{align*}
Using Theorem \ref{thm:erroruJq}, Lemma \ref{lem:est_du} and the
inf-sup condition, 
\[
\beta_{m}\left\Vert e_{\phi}^{n+1}\right\Vert \leq\sup_{\boldsymbol{K}\in\boldsymbol{D}}\frac{\left(e_{\phi}^{n+1},{\rm div}\boldsymbol{K}\right)}{\|\boldsymbol{K}\|_{{\rm div}}},
\]

\noindent we have
\begin{align*}
\tau\sum_{n=0}^{m}\left\Vert e_{\phi}^{n+1}\right\Vert ^{2} & \le C\left(\tau\sum_{n=0}^{m}\left\Vert e_{\boldsymbol{J}}^{n+1}\right\Vert ^{2}+\tau\sum_{n=0}^{m}\left|e_{q}^{n+1}\right|^{2}+\tau\sum_{n=0}^{m}\left\Vert \nabla e_{\boldsymbol{u}}^{n}\right\Vert ^{2}+\tau\int_{0}^{t^{m}}\left\Vert \nabla\boldsymbol{u}_{t}\right\Vert ^{2}{\rm d}t\right)\\
 & \le C\tau^{2}.
\end{align*}
This completes the proof.
\end{proof}
\medskip{}

To derive error estimate for the pressure, we need to establish the
estimate for $\delta_{t}e_{\boldsymbol{u}}^{n+1}$.
\begin{lem}
\label{lem:est_du}Assuming $\boldsymbol{u}\in H^{2}(0,T;\boldsymbol{H}^{2}(\Omega))\bigcap H^{1}(0,T;\boldsymbol{H}^{2}(\Omega))\bigcap L^{\infty}(0,T;\boldsymbol{H}^{2}(\Omega))$,
$\boldsymbol{J}\in L^{\infty}(0,T;\boldsymbol{L}^{2}(\Omega))\cap H^{1}(0,T;\boldsymbol{L}^{2}(\Omega))$
and $\boldsymbol{B}\in L^{\infty}(0,T;\boldsymbol{L}^{\infty}(\Omega))$,
then we have the following error estimate for $\ 0\leq m\leq N-1$,
\begin{equation}
\left\Vert \nabla e_{\boldsymbol{u}}^{m+1}\right\Vert ^{2}+\tau\sum\limits _{n=0}^{m}\left\Vert \delta_{t}e_{\boldsymbol{u}}^{n+1}\right\Vert ^{2}+\nu\tau\sum\limits _{n=0}^{m}\left\Vert Ae_{\boldsymbol{u}}^{n+1}\right\Vert ^{2}\le C\tau^{2}.\label{eq:est_du}
\end{equation}
\end{lem}
\begin{proof}
First of all, in virtue of (\ref{eq:estuJq}), we have
\[
\left\Vert \nabla e_{\boldsymbol{u}}^{n+1}\right\Vert ^{2}\le\tau^{-1}\left(\tau\sum_{k=0}^{m}\left\Vert \nabla e_{\boldsymbol{u}}^{k+1}\right\Vert ^{2}\right)\leq C\tau,\quad\left\Vert e_{\boldsymbol{J}}^{n+1}\right\Vert ^{2}\le\tau^{-1}\left(\tau\sum_{k=0}^{m}\left\Vert e_{\boldsymbol{J}}^{k+1}\right\Vert ^{2}\right)\leq C\tau.
\]
Hence, there holds that
\begin{align}
\left\Vert \nabla\boldsymbol{u}^{n+1}\right\Vert  & \leq\left\Vert \nabla e_{\boldsymbol{u}}^{n+1}\right\Vert +\left\Vert \nabla\boldsymbol{u}\left(t^{n+1}\right)\right\Vert \leq C\left(\tau^{1/2}+\left\Vert \nabla\boldsymbol{u}\left(t^{n+1}\right)\right\Vert \right),\label{eq:gradu}\\
\left\Vert \boldsymbol{J}^{n+1}\right\Vert  & \leq\left\Vert e_{\boldsymbol{J}}^{n+1}\right\Vert +\left\Vert \boldsymbol{J}\left(t^{n+1}\right)\right\Vert \leq C\left(\tau^{1/2}+\left\Vert \boldsymbol{J}\left(t^{n+1}\right)\right\Vert \right).\label{eq:J}
\end{align}
Taking the inner product of (\ref{eq:erroru}) with $Ae_{\boldsymbol{u}}^{n+1}+\delta_{t}e_{\boldsymbol{u}}^{n+1}$,
we obtain 
\begin{align}
 & (1+R_{e}^{-1})\frac{\left\Vert \nabla e_{\boldsymbol{u}}^{n+1}\right\Vert ^{2}-\left\Vert \nabla e_{\boldsymbol{u}}^{n}\right\Vert ^{2}+\left\Vert \nabla e_{\boldsymbol{u}}^{n+1}-\nabla e_{\boldsymbol{u}}^{n}\right\Vert ^{2}}{2\tau}+\left\Vert \delta_{t}e_{\boldsymbol{u}}^{n+1}\right\Vert ^{2}+R_{e}^{-1}\left\Vert Ae_{\boldsymbol{u}}^{n+1}\right\Vert ^{2}\nonumber \\
 & =\left(R_{\boldsymbol{u}}^{n+1},Ae_{\boldsymbol{u}}^{n+1}+\delta_{t}e_{\boldsymbol{u}}^{n+1}\right)+\kappa\exp\left(\frac{t^{n+1}}{T}\right)\left(q^{n+1}\boldsymbol{J}^{n}\times\boldsymbol{B}^{n+1}-q\left(t^{n+1}\right)\boldsymbol{J}\left(t^{n+1}\right)\times\boldsymbol{B}^{n+1},Ae_{\boldsymbol{u}}^{n+1}+\delta_{t}e_{\boldsymbol{u}}^{n+1}\right)\nonumber \\
 & \quad+\exp\left(\frac{t^{n+1}}{T}\right)\left(q\left(t^{n+1}\right)\boldsymbol{u}\left(t^{n+1}\right)\cdot\nabla\boldsymbol{u}\left(t^{n+1}\right)-q^{n+1}\boldsymbol{u}^{n}\cdot\nabla\boldsymbol{u}^{n},Ae_{\boldsymbol{u}}^{n+1}+\delta_{t}e_{\boldsymbol{u}}^{n+1}\right)\nonumber \\
 & \coloneqq\sum_{i=1}^{3}{\rm I}_{i}.\label{eq:e_error_Au}
\end{align}

For term ${\rm I}_{1}$, we use Cauchy-Schwarz and Young inequality
to estimate it as
\begin{equation}
\left(R_{\boldsymbol{u}}^{n+1},Ae_{\boldsymbol{u}}^{n+1}+\delta_{t}e_{\boldsymbol{u}}^{n+1}\right)\leq\frac{1}{12}\left\Vert \delta_{t}e_{\boldsymbol{u}}^{n+1}\right\Vert ^{2}+\frac{\nu}{24}\left\Vert Ae_{\boldsymbol{u}}^{n+1}\right\Vert ^{2}+C\tau\int_{t^{n}}^{t^{n+1}}\left\Vert \boldsymbol{u}_{tt}\right\Vert ^{2}dt.\label{eq:e_error_Au_Ru}
\end{equation}

For term ${\rm I}_{2}$, using Hölder inequality and Young inequality,
we obtain
\begin{align}
{\rm I}_{2} & =\exp\left(\frac{t^{n+1}}{T}\right)e_{q}^{n+1}\left(\boldsymbol{J}^{n}\times\boldsymbol{B}^{n+1},Ae_{\boldsymbol{u}}^{n+1}+\delta_{t}e_{\boldsymbol{u}}^{n+1}\right)+\left(\left(\boldsymbol{J}^{n}-\boldsymbol{J}\left(t^{n+1}\right)\right)\times\boldsymbol{B}^{n+1},Ae_{\boldsymbol{u}}^{n+1}+\delta_{t}e_{\boldsymbol{u}}^{n+1}\right)\nonumber \\
 & =\exp\left(\frac{t^{n+1}}{T}\right)e_{q}^{n+1}\left(\boldsymbol{J}^{n}\times\boldsymbol{B}^{n+1},Ae_{\boldsymbol{u}}^{n+1}+\delta_{t}e_{\boldsymbol{u}}^{n+1}\right)+\left(e_{\boldsymbol{J}}^{n}\times\boldsymbol{B}^{n+1},Ae_{\boldsymbol{u}}^{n+1}+\delta_{t}e_{\boldsymbol{u}}^{n+1}\right)\nonumber \\
 & \quad+\left(\left(\boldsymbol{J}\left(t^{n}\right)-\boldsymbol{J}\left(t^{n+1}\right)\right)\times\boldsymbol{B}^{n+1},Ae_{\boldsymbol{u}}^{n+1}+\delta_{t}e_{\boldsymbol{u}}^{n+1}\right)\nonumber \\
 & \le\exp(1)\left|e_{q}^{n+1}\right|\left\Vert \boldsymbol{J}^{n}\right\Vert \left\Vert \boldsymbol{B}^{n+1}\right\Vert _{0,\infty}\left\Vert Ae_{\boldsymbol{u}}^{n+1}+\delta_{t}e_{\boldsymbol{u}}^{n+1}\right\Vert +\left\Vert e_{\boldsymbol{J}}^{n}\right\Vert \left\Vert \boldsymbol{B}^{n+1}\right\Vert _{0,\infty}\left\Vert Ae_{\boldsymbol{u}}^{n+1}+\delta_{t}e_{\boldsymbol{u}}^{n+1}\right\Vert \nonumber \\
 & \quad+\left\Vert \int_{t^{n}}^{t^{n+1}}\boldsymbol{J}_{t}{\rm d}t\right\Vert \left\Vert \boldsymbol{B}^{n+1}\right\Vert _{0,\infty}\left\Vert Ae_{\boldsymbol{u}}^{n+1}+\delta_{t}e_{\boldsymbol{u}}^{n+1}\right\Vert \nonumber \\
 & \le\frac{1}{6}\left\Vert \delta_{t}e_{\boldsymbol{u}}^{n+1}\right\Vert ^{2}+\frac{\nu}{12}\left\Vert Ae_{\boldsymbol{u}}^{n+1}\right\Vert ^{2}+C\left(\tau+\left\Vert \boldsymbol{J}\left(t^{n+1}\right)\right\Vert ^{2}\right)\left\Vert \boldsymbol{B}^{n+1}\right\Vert _{0,\infty}^{2}\left|e_{q}^{n+1}\right|\nonumber \\
 & \quad+C\left\Vert \boldsymbol{B}^{n+1}\right\Vert _{0,\infty}^{2}\left\Vert e_{\boldsymbol{J}}^{n}\right\Vert ^{2}+C\tau\left\Vert \boldsymbol{B}^{n+1}\right\Vert _{0,\infty}^{2}\int_{t^{n}}^{t^{n+1}}\left\Vert \boldsymbol{J}_{t}\right\Vert ^{2}{\rm d}t.\label{eq:e_error_Au_J}
\end{align}

For term ${\rm I}_{3}$, we rearrange it as follows
\begin{equation}
\begin{aligned}{\rm I}_{3} & =-\exp\left(\frac{t^{n+1}}{T}\right)e_{q}^{n+1}\left(\boldsymbol{u}^{n}\cdot\nabla\boldsymbol{u}^{n},Ae_{\boldsymbol{u}}^{n+1}+\delta_{t}e_{\boldsymbol{u}}^{n+1}\right)+\left(\left(\boldsymbol{u}\left(t^{n+1}\right)-\boldsymbol{u}^{n}\right)\cdot\nabla\boldsymbol{u}\left(t^{n+1}\right),Ae_{\boldsymbol{u}}^{n+1}+\delta_{t}e_{\boldsymbol{u}}^{n+1}\right)\\
 & \quad+\left(\boldsymbol{u}^{n}\cdot\nabla\left(\boldsymbol{u}\left(t^{n+1}\right)-\boldsymbol{u}^{n}\right),Ae_{\boldsymbol{u}}^{n+1}+\delta_{t}e_{\boldsymbol{u}}^{n+1}\right)\\
 & \coloneqq\sum_{i=1}^{3}{\rm I}_{3,i}.
\end{aligned}
\label{eq:e_error_Au_nonlinear1}
\end{equation}
Term ${\rm I}_{3,1}$ can be bounded by using (\ref{eq:e_trilinear_form2d1}),
(\ref{eq:e_trilinear_form4}) and (\ref{eq:gradu}), the first term
on the right hand side of (\ref{eq:e_error_Au_nonlinear1}) can be
bounded by 
\begin{align}
{\rm I}_{3,1} & =-\exp\left(\frac{t^{n+1}}{T}\right)e_{q}^{n+1}\left(\boldsymbol{u}^{n}\cdot\nabla e_{\boldsymbol{u}}^{n},Ae_{\boldsymbol{u}}^{n+1}+\delta_{t}e_{\boldsymbol{u}}^{n+1}\right)-\exp\left(\frac{t^{n+1}}{T}\right)e_{q}^{n+1}\left((\boldsymbol{u}^{n}\cdot\nabla\boldsymbol{u}(t^{n}),Ae_{\boldsymbol{u}}^{n+1}+\delta_{t}e_{\boldsymbol{u}}^{n+1}\right)\nonumber \\
 & \leq\exp(1)C_{b,7}|e_{q}^{n+1}|\left\Vert \boldsymbol{u}^{n}\right\Vert ^{1/2}\left\Vert \nabla\boldsymbol{u}^{n}\right\Vert ^{1/2}\left\Vert \nabla e_{\boldsymbol{u}}^{n}\right\Vert ^{1/2}\left\Vert Ae_{\boldsymbol{u}}^{n}\right\Vert ^{1/2}\left\Vert Ae_{\boldsymbol{u}}^{n+1}+\delta_{t}e_{\boldsymbol{u}}^{n+1}\right\Vert \nonumber \\
 & \quad+\exp(1)C_{b,5}|e_{q}^{n+1}|\left\Vert \nabla\boldsymbol{u}^{n}\right\Vert \left\Vert \boldsymbol{u}(t^{n})\right\Vert _{2}\left\Vert Ae_{\boldsymbol{u}}^{n+1}+\delta_{t}e_{\boldsymbol{u}}^{n+1}\right\Vert \nonumber \\
 & \leq\frac{1}{12}\left\Vert \delta_{t}e_{\boldsymbol{u}}^{n+1}\right\Vert ^{2}+\frac{\nu}{24}\left\Vert Ae_{\boldsymbol{u}}^{n+1}\right\Vert ^{2}+\frac{\nu}{8}\left\Vert Ae_{\boldsymbol{u}}^{n}\right\Vert ^{2}\nonumber \\
 & \quad+C\left(\tau+\left\Vert \nabla\boldsymbol{u}(t^{n})\right\Vert ^{2}\right)\left\Vert \nabla e_{\boldsymbol{u}}^{n}\right\Vert ^{2}+C\left(\tau+\left\Vert \nabla\boldsymbol{u}(t^{n})\right\Vert ^{2}\right)\left\Vert \boldsymbol{u}(t^{n})\right\Vert _{2}^{2}|e_{q}^{n+1}|^{2}.\label{eq:e_error_Au_nonlinear2}
\end{align}
Similarly, term ${\rm I}_{3,2}$ can be estimated by 
\begin{align}
{\rm I}_{3,2} & \leq C_{b,5}\left\Vert \nabla\boldsymbol{u}(t^{n+1})-\nabla\boldsymbol{u}^{n}\right\Vert \left\Vert \boldsymbol{u}(t^{n+1})\right\Vert _{2}\left\Vert Ae_{\boldsymbol{u}}^{n+1}+\delta_{t}e_{\boldsymbol{u}}^{n+1}\right\Vert \nonumber \\
 & \leq\frac{1}{12}\left\Vert \delta_{t}e_{\boldsymbol{u}}^{n+1}\right\Vert ^{2}+\frac{\nu}{24}\left\Vert Ae_{\boldsymbol{u}}^{n+1}\right\Vert ^{2}+C\left\Vert \boldsymbol{u}(t^{n+1})\right\Vert _{2}^{2}\left\Vert \nabla e_{\boldsymbol{u}}^{n}\right\Vert ^{2}+C\left\Vert \boldsymbol{u}(t^{n+1})\right\Vert _{2}^{2}\tau\int_{t^{n}}^{t^{n+1}}\left\Vert \nabla\boldsymbol{u}_{t}(s)\right\Vert ^{2}ds,\label{eq:e_error_Au_nonlinear3}
\end{align}
For term ${\rm I}_{3,3}$, we deduce that
\begin{align}
{\rm I}_{3,3} & =\left(\boldsymbol{u}^{n}\cdot\nabla(\boldsymbol{u}(t^{n+1})-\boldsymbol{u}(t^{n})),Ae_{\boldsymbol{u}}^{n+1}+\delta_{t}e_{\boldsymbol{u}}^{n+1}\right)-\left(\boldsymbol{u}^{n}\cdot\nabla e_{\boldsymbol{u}}^{n},Ae_{\boldsymbol{u}}^{n+1}+\delta_{t}e_{\boldsymbol{u}}^{n+1}\right)\nonumber \\
 & \leq C_{b,5}\left\Vert \nabla\boldsymbol{u}^{n}\right\Vert \left\Vert \boldsymbol{u}(t^{n+1})-\boldsymbol{u}(t^{n})\right\Vert _{2}\left\Vert Ae_{\boldsymbol{u}}^{n+1}+\delta_{t}e_{\boldsymbol{u}}^{n+1}\right\Vert \nonumber \\
 & \quad+C_{b,6}\left\Vert \boldsymbol{u}^{n}\right\Vert ^{1/2}\left\Vert \nabla\boldsymbol{u}^{n}\right\Vert ^{1/2}\left\Vert \nabla e_{\boldsymbol{u}}^{n}\right\Vert ^{1/2}\left\Vert Ae_{\boldsymbol{u}}^{n}\right\Vert ^{1/2}\left\Vert Ae_{\boldsymbol{u}}^{n+1}+\delta_{t}e_{\boldsymbol{u}}^{n+1}\right\Vert \nonumber \\
 & \leq\frac{1}{12}\left\Vert \delta_{t}e_{\boldsymbol{u}}^{n+1}\right\Vert ^{2}+\frac{\nu}{24}\left\Vert Ae_{\boldsymbol{u}}^{n+1}\right\Vert ^{2}+C\left(\tau+\left\Vert \nabla\boldsymbol{u}(t^{n})\right\Vert ^{2}\right)\left\Vert \nabla e_{\boldsymbol{u}}^{n}\right\Vert ^{2}\nonumber \\
 & \quad+\frac{\nu}{8}\left\Vert Ae_{\boldsymbol{u}}^{n}\right\Vert ^{2}+C\tau\left(\tau+\left\Vert \nabla\boldsymbol{u}(t^{n})\right\Vert ^{2}\right)\int_{t^{n}}^{t^{n+1}}\left\Vert \boldsymbol{u}_{t}(s)\right\Vert _{2}^{2}ds.\label{eq:e_error_Au_nonlinear4}
\end{align}
Combining (\ref{eq:e_error_Au}) with (\ref{eq:e_error_Au_Ru})-(\ref{eq:e_error_Au_nonlinear1}),
we have 
\begin{align}
 & (1+R_{e}^{-1})\frac{\left\Vert \nabla e_{\boldsymbol{u}}^{n+1}\right\Vert ^{2}-\left\Vert \nabla e_{\boldsymbol{u}}^{n}\right\Vert ^{2}+\left\Vert \nabla e_{\boldsymbol{u}}^{n+1}-\nabla e_{\boldsymbol{u}}^{n}\right\Vert ^{2}}{2\tau}+\left\Vert \delta_{t}e_{\boldsymbol{u}}^{n+1}\right\Vert ^{2}+R_{e}^{-1}\left\Vert Ae_{\boldsymbol{u}}^{n+1}\right\Vert ^{2}\nonumber \\
 & =\frac{R_{e}^{-1}}{4}\left\Vert Ae_{\boldsymbol{u}}^{n}\right\Vert ^{2}+C\left(1+\tau+\left\Vert \nabla\boldsymbol{u}(t^{n})\right\Vert ^{2}\right)\left(\left\Vert \nabla e_{\boldsymbol{u}}^{n}\right\Vert ^{2}+|e_{q}^{n+1}|^{2}\right)\nonumber \\
 & \quad+C\tau\int_{t^{n}}^{t^{n+1}}\left(\left\Vert \boldsymbol{u}_{t}(s)\right\Vert _{2}^{2}ds+\left\Vert \nabla\boldsymbol{u}_{t}(s)\right\Vert ^{2}+\left\Vert \boldsymbol{u}_{tt}(s)\right\Vert ^{2}\right)ds.\label{eq:e_error_Au_final1}
\end{align}
Multiplying (\ref{eq:e_error_Au_final1}) by $2\tau$ and summing
over $n$ from 0 to $m$, and applying the discrete Gronwall inequality
in Lemma \ref{lem: gronwall2}, we obtain 
\begin{align}
 & \left\Vert \nabla e_{\boldsymbol{u}}^{m+1}\right\Vert ^{2}+\tau\sum\limits _{n=0}^{m}\left\Vert \delta_{t}e_{\boldsymbol{u}}^{n+1}\right\Vert ^{2}+\tau R_{e}^{-1}\sum\limits _{n=0}^{m}\left\Vert Ae_{\boldsymbol{u}}^{n+1}\right\Vert ^{2}\nonumber \\
 & \le C\left(1+\tau+\left\Vert \nabla\boldsymbol{u}(t^{n})\right\Vert ^{2}\right)\tau\sum\limits _{n=0}^{m}\left(\left\Vert \nabla e_{\boldsymbol{u}}^{n}\right\Vert ^{2}+|e_{q}^{n+1}|^{2}\right)+C\tau^{2}.\label{eq:e_error_Au_final2}
\end{align}
Combining the above estimate with Theorem \ref{eq:estuJq}, we obtain
the desired result.
\end{proof}
\medskip{}

We are now in position to prove the pressure estimate. 
\begin{thm}
\label{thm:error_estimate_p}Assuming $\boldsymbol{u}\in H^{2}(0,T;\boldsymbol{H}^{2}(\Omega))\bigcap H^{1}(0,T;\boldsymbol{H}^{2}(\Omega))\bigcap L^{\infty}(0,T;\boldsymbol{H}^{2}(\Omega))$,
$\boldsymbol{J}\in L^{\infty}(0,T;\boldsymbol{L}^{2}(\Omega))\cap H^{1}(0,T;\boldsymbol{L}^{2}(\Omega))$
and $\boldsymbol{B}\in L^{\infty}(0,T;\boldsymbol{L}^{\infty}(\Omega))$,
then we have the following error estimate for $\ 0\leq m\leq N-1$,
\begin{equation}
\tau\sum_{n=0}^{m}\left\Vert e_{p}^{n+1}\right\Vert ^{2}\le C\tau^{2}.\label{eq:est_p}
\end{equation}
\end{thm}
\begin{proof}
Taking the inner product of (\ref{eq:erroru}) with $\boldsymbol{v}\in\boldsymbol{X}$,
we obtain 
\begin{align}
\left(\nabla e_{p}^{n+1},\boldsymbol{v}\right)= & -\left(\delta_{t}e_{\boldsymbol{u}}^{n+1},\boldsymbol{v}\right)+R_{e}^{-1}\left(\nabla e_{\boldsymbol{u}}^{n+1},\boldsymbol{v}\right)+\left(R_{\boldsymbol{u}}^{n+1},\boldsymbol{v}\right)\nonumber \\
 & +\exp\left(\frac{t^{n+1}}{T}\right)\left(q\left(t^{n+1}\right)\left(\boldsymbol{u}\left(t^{n+1}\right)\cdot\nabla\right)\boldsymbol{u}\left(t^{n+1}\right)-q^{n+1}\left(\boldsymbol{u}^{n}\cdot\nabla\right)\boldsymbol{u}^{n},\boldsymbol{v}\right)\nonumber \\
 & +\kappa\exp\left(\frac{t^{n+1}}{T}\right)\left(q^{n+1}\boldsymbol{J}^{n}\times\boldsymbol{B}^{n+1}-q\left(t^{n+1}\right)\boldsymbol{J}\left(t^{n+1}\right)\times\boldsymbol{B}^{n+1},\boldsymbol{v}\right)\label{eq:e_error_p1}
\end{align}
It is easy to see that the first three terms can be easily bounded
by
\begin{equation}
-\left(\delta_{t}e_{\boldsymbol{u}}^{n+1},\boldsymbol{v}\right)+R_{e}^{-1}\left(\nabla e_{\boldsymbol{u}}^{n+1},\boldsymbol{v}\right)+\left(R_{\boldsymbol{u}}^{n+1},\boldsymbol{v}\right)\le C\left(\left\Vert \delta_{t}e_{\boldsymbol{u}}^{n+1}\right\Vert +\left\Vert \nabla e_{\boldsymbol{u}}^{n+1}\right\Vert +\left\Vert R_{\boldsymbol{u}}^{n+1}\right\Vert _{-1}\right)\left\Vert \nabla\boldsymbol{v}\right\Vert .\label{eq:e_error_p0}
\end{equation}
For the forth term, by using (\ref{eq:e_trilinear_form})-(\ref{eq:e_trilinear_form2d})
and (\ref{eq:gradu}), we have that for all $\boldsymbol{v}\in\boldsymbol{X}$,
\begin{align}
 & \begin{aligned}\exp\end{aligned}
(\frac{t^{n+1}}{T})\left(q(t^{n+1})\boldsymbol{u}(t^{n+1})\cdot\nabla\boldsymbol{u}(t^{n+1})-q^{n+1}\boldsymbol{u}^{n}\cdot\nabla\boldsymbol{u}^{n},\boldsymbol{v}\right)\nonumber \\
 & =\left((\boldsymbol{u}(t^{n+1})-\boldsymbol{u}^{n})\cdot\nabla\boldsymbol{u}(t^{n+1}),\boldsymbol{v}\right)-e_{q}^{n+1}\begin{aligned}\exp\end{aligned}
(\frac{t^{n+1}}{T})\left(\boldsymbol{u}^{n}\cdot\nabla\boldsymbol{u}^{n},\boldsymbol{v}\right)+\left(\boldsymbol{u}^{n}\cdot\nabla(\boldsymbol{u}(t^{n+1})-\boldsymbol{u}^{n}),\boldsymbol{v}\right)\nonumber \\
 & \le C_{b,1}\left\Vert \nabla\boldsymbol{u}(t^{n+1})-\nabla\boldsymbol{u}^{n}\right\Vert \left\Vert \boldsymbol{u}(t^{n+1})\right\Vert _{2}\left\Vert \nabla\boldsymbol{v}\right\Vert +C\left|e_{q}^{n+1}\right|\left\Vert \nabla\boldsymbol{u}^{n}\right\Vert \left\Vert \nabla\boldsymbol{u}^{n}\right\Vert \left\Vert \nabla\boldsymbol{v}\right\Vert \nonumber \\
 & \quad+C\left\Vert \nabla\boldsymbol{u}^{n}\right\Vert \left\Vert \nabla\boldsymbol{u}(t^{n+1})-\nabla\boldsymbol{u}^{n}\right\Vert \left\Vert \nabla\boldsymbol{v}\right\Vert \nonumber \\
 & \leq C\left(\left\Vert \nabla e_{\boldsymbol{u}}^{n}\right\Vert +\left\Vert \int_{t^{n}}^{t^{n+1}}\nabla\boldsymbol{u}_{t}\mathrm{d}t\right\Vert +|e_{q}^{n+1}|\right)\left\Vert \nabla\boldsymbol{v}\right\Vert .\label{eq:e_error_p2}
\end{align}
For the last term, we invoke to estimate it as we have that for all
$\boldsymbol{v}\in\boldsymbol{X}$, 
\begin{align}
 & \kappa\exp\left(\frac{t^{n+1}}{T}\right)\left(q^{n+1}\boldsymbol{J}^{n}\times\boldsymbol{B}^{n+1}-q\left(t^{n+1}\right)\boldsymbol{J}\left(t^{n+1}\right)\times\boldsymbol{B}^{n+1},\boldsymbol{v}\right)\nonumber \\
 & =\exp\left(\frac{t^{n+1}}{T}\right)e_{q}^{n+1}\left(\boldsymbol{J}^{n}\times\boldsymbol{B}^{n+1},\boldsymbol{v}\right)+\left(\left(\boldsymbol{J}^{n}-\boldsymbol{J}\left(t^{n+1}\right)\right)\times\boldsymbol{B}^{n+1},\boldsymbol{v}\right)\nonumber \\
 & =\exp\left(\frac{t^{n+1}}{T}\right)e_{q}^{n+1}\left(\boldsymbol{J}^{n}\times\boldsymbol{B}^{n+1},\boldsymbol{v}\right)+\left(e_{\boldsymbol{J}}^{n}\times\boldsymbol{B}^{n+1},\boldsymbol{v}\right)+\left(\left(\boldsymbol{J}\left(t^{n}\right)-\boldsymbol{J}\left(t^{n+1}\right)\right)\times\boldsymbol{B}^{n+1},\boldsymbol{v}\right)\nonumber \\
 & \le\exp(1)\left|e_{q}^{n+1}\right|\left\Vert \boldsymbol{J}^{n}\right\Vert \left\Vert \boldsymbol{B}^{n+1}\right\Vert _{0,\infty}\left\Vert \boldsymbol{v}\right\Vert +\left\Vert e_{\boldsymbol{J}}^{n}\right\Vert \left\Vert \boldsymbol{B}^{n+1}\right\Vert _{0,\infty}\left\Vert \boldsymbol{v}\right\Vert +\left\Vert \int_{t^{n}}^{t^{n+1}}\boldsymbol{J}_{t}{\rm d}t\right\Vert \left\Vert \boldsymbol{B}^{n+1}\right\Vert _{0,\infty}\left\Vert \boldsymbol{v}\right\Vert \nonumber \\
 & \le C\left(\left|e_{q}^{n+1}\right|+\left\Vert e_{\boldsymbol{J}}^{n}\right\Vert +\left\Vert \int_{t^{n}}^{t^{n+1}}\boldsymbol{J}_{t}{\rm d}t\right\Vert \right)\left\Vert \nabla\boldsymbol{v}\right\Vert .\label{eq:error_p3}
\end{align}
Using Theorem \ref{thm:erroruJq}, Lemma \ref{lem:est_du} and the
inf-sup condition, 
\begin{equation}
\beta_{s}\left\Vert e_{p}^{n+1}\right\Vert \leq\sup_{\boldsymbol{v}\in\boldsymbol{X}}\frac{\left(\nabla e_{p}^{n+1},\boldsymbol{v}\right)}{\|\nabla\boldsymbol{v}\|},\label{eq:e_error_p4}
\end{equation}
we get
\begin{align*}
\tau\sum_{n=0}^{m}\left\Vert e_{p}^{n+1}\right\Vert ^{2} & \le C\tau\sum_{n=0}^{m}\left(\left\Vert \delta_{t}e_{\boldsymbol{u}}^{n+1}\right\Vert ^{2}+R_{e}^{-1}\left\Vert \nabla e_{\boldsymbol{u}}^{n+1}\right\Vert ^{2}+\left\Vert \nabla e_{\boldsymbol{u}}^{n}\right\Vert ^{2}+|e_{q}^{n+1}|^{2}+\left\Vert e_{\boldsymbol{J}}^{n}\right\Vert ^{2}\right)\\
 & \quad+C\tau^{2}\int_{0}^{t^{m+1}}\left(\left\Vert \nabla\boldsymbol{u}_{t}\right\Vert ^{2}+\left\Vert \boldsymbol{u}_{tt}\right\Vert _{-1}^{2}+\left\Vert \boldsymbol{J}_{t}\right\Vert ^{2}\right)dt\\
 & \le C\tau^{2}.
\end{align*}
The proof is complete. 
\end{proof}
\begin{rem}
\label{rem:pressu}In this paper, we only focus on designing unconditionally
energy-stable and linear SAV schemes for the inductionless MHD equations.
The velocity and pressure can be further decoupled by using the classical
pressure correction scheme \citep{Li2020b,Yang2021m,Guermond2006},
and we leave it to the interested readers.
\end{rem}

\section{Numerical experiments\label{sec:Num}}

In this section, we present a series of numerical experiments to verify
the theoretical results of the proposed schemes. The numerical experiments
are implemented on the finite element software FreeFEM \citep{Hecht2012}. 

\subsection{The fully-discrete schemes}

Although we only discussed semi-discretization in time in the previous
sections, the SAV schemes can be coupled with any compatible spatial
discretization. In this work, the spatial discretization is based
on mixed finite element method. 

Let $\mathcal{T}_{h}$ be a quasi-uniform and shape-regular tetrahedral
mesh of $\Omega$. As usual, we introduce the local mesh size $h_{K}=\mathrm{diam}\left(K\right)$
and the global mesh size $h:=\underset{K\in\mathcal{T}_{h}}{\max}h_{K}$.
For any integer $k\geq0,$ let $P_{k}(K)$ be the space of polynomials
of degree $k$ on element $K$ and define $\boldsymbol{P}_{k}(K)=P_{k}(K)^{3}$.
Following \citep{Zhang2021,Li19,Zhang2020}, we employ the Mini-element
to approximate the velocity and pressure
\[
\boldsymbol{X}_{h}=\boldsymbol{P}_{1,h}^{b}\cap\boldsymbol{X},\quad Y_{h}=\left\{ r_{h}\in H^{1}(\Omega):\left.r_{h}\right|_{K}\in P_{1}(K),\,\forall K\in\mathcal{T}_{h}\right\} \cap Y,
\]
where $P_{1,h}^{b}=\left\{ v_{h}\in C^{0}(\Omega):\left.v_{h}\right|_{K}\in P_{1}(K)\oplus{\rm span}\{\hat{b}\},\,\forall K\in\mathcal{T}_{h}\right\} $,
$\hat{b}$ is a bubble function on $K$. We choose the lowest-order Raviart-Thomas
element space given by
\[
\boldsymbol{D}_{h}=\left\{ \boldsymbol{K}_{h}\in\boldsymbol{D}:\left.\boldsymbol{K}_{h}\right|_{K}\in\boldsymbol{P}_{0}(K)+\boldsymbol{x}P_{0}(K),\,\forall K\in\mathcal{T}_{h}\right\} ,
\]
combined with the discontinuous and piece-wise constant finite element
space
\[
S_{h}=\left\{ \psi_{h}\in L^{2}(\Omega):\left.\psi_{h}\right|_{K}\in P_{0}(K),\,\forall K\in\mathcal{T}_{h}\right\} \cap S.
\]
From \citep{Gir86,Brezzi2012,John16}, the two finite element pairs,
$(\boldsymbol{X}_{h},Y_{h})$ and $(\boldsymbol{D}_{h},S_{h})$ ,
satisfy the following uniform inf-sup conditions, 
\begin{equation}
\inf_{0\ne q_{h}\in Q_{h}}\sup_{\boldsymbol{0}\ne\boldsymbol{v}_{h}\in\boldsymbol{V}_{h}}\frac{b_{s}(q_{h},\boldsymbol{v}_{h})}{\left\Vert \nabla\boldsymbol{v}_{h}\right\Vert _{0}\left\Vert q_{h}\right\Vert }\geq\beta_{s},\quad\inf_{0\ne\psi_{h}\in S_{h}}\sup_{\boldsymbol{0}\ne\boldsymbol{K}_{h}\in\boldsymbol{D}_{h}}\frac{b_{m}(\psi_{h},\boldsymbol{K}_{h})}{\left\Vert \boldsymbol{K}_{h}\right\Vert _{{\rm _{div}}}\left\Vert \psi_{h}\right\Vert }\geq\beta_{m},\label{bsbm:inf-suph}
\end{equation}
where $\beta_{s}$ and $\beta_{m}$ are constants independent of the
mesh size. 

Based on (\ref{eq:SAVEuler}), a fully discrete first-order SAV scheme
is as follows. For all $n\ge0$, we find $\left(\boldsymbol{u}_{h}^{n+1},p_{h}^{n+1},\boldsymbol{J}_{h}^{n+1},\phi_{h}^{n+1}\right)\in\boldsymbol{X}_{h}\times Y_{h}\times\boldsymbol{D}_{h}\times S_{h}$
and $q_{h}\in\mathbb{R}$ such that for all $\left(\boldsymbol{v}_{h},r_{h},\boldsymbol{K}_{h},\psi_{h}\right)\in\boldsymbol{X}_{h}\times Y_{h}\times\boldsymbol{D}_{h}\times S_{h}$,
\begin{subequations}
\begin{align}
(\delta_{t}\boldsymbol{u}_{h}^{n+1},\boldsymbol{v}_{h})+R_{e}^{-1}(\nabla\boldsymbol{u}_{h}^{n+1},\nabla\boldsymbol{v}_{h})-(p^{n+1},\nabla\cdot\boldsymbol{v}_{h})\nonumber \\
+q_{h}^{n+1}\exp\left(\frac{t^{n+1}}{T}\right)\left(\left(\boldsymbol{u}_{h}^{n}\cdot\nabla\boldsymbol{u}_{h}^{n},\boldsymbol{v}_{h}\right)-\kappa\left(\boldsymbol{J}_{h}^{n}\times\boldsymbol{B}^{n+1},\boldsymbol{v}_{h}\right)\right) & =0,\label{eq:weakhu}\\
\left(\nabla\cdot\boldsymbol{u}_{h}^{n+1},r_{h}\right) & =0,\label{eq:weakhdivu}\\
\left(\boldsymbol{J}_{h}^{n+1},\boldsymbol{K}_{h}\right)-\left(\phi_{h}^{n+1},\nabla\cdot\boldsymbol{K}_{h}\right)-q_{h}^{n+1}\exp\left(\frac{t^{n+1}}{T}\right)\left(\boldsymbol{u}_{h}^{n}\times\boldsymbol{B}^{n+1},\boldsymbol{K}_{h}\right) & =0,\label{eq:weakhJ}\\
\left(\nabla\cdot\boldsymbol{J}_{h}^{n+1},\psi_{h}\right) & =0,\label{eq:weakhdivJ}\\
\delta_{t}q_{h}^{n+1}+\frac{q_{h}^{n+1}}{T}-\exp\left(\frac{t^{n+1}}{T}\right)\nonumber \\
\left(\left(\boldsymbol{u}_{h}^{n}\cdot\nabla\boldsymbol{u}_{h}^{n},\boldsymbol{u}_{h}^{n+1}\right)-\kappa\left(\boldsymbol{u}_{h}^{n}\times\boldsymbol{B}^{n+1},\boldsymbol{J}_{h}^{n+1}\right)-\kappa\left(\boldsymbol{J}_{h}^{n}\times\boldsymbol{B}^{n+1},\boldsymbol{u}_{h}^{n+1}\right)\right) & =0.\label{eq:weakhq}
\end{align}
\label{eq:SAVEulerh}
\end{subequations}
For convenience, the init data $q_{h}^{0}=q^{0}=1$,
$\boldsymbol{u}_{h}^{0}$ is taken as the standard interpolation
of $\boldsymbol{u}^{0}$ onto $\boldsymbol{X}_{h}$, $\boldsymbol{J}_{h}^{0}$
is obtained by solving discrete problem, for all $\left(\boldsymbol{K}_{h},\psi_{h}\right)\in\boldsymbol{D}_{h}\times S_{h}$
, 

\begin{align*}
\left(\boldsymbol{J}_{h}^{0},\boldsymbol{K}_{h}\right)-\left(\phi_{h}^{0},\nabla\cdot\boldsymbol{K}_{h}\right)-\left(\boldsymbol{u}^{0}\times\boldsymbol{B}^{0},\boldsymbol{K}_{h}\right) & =0,\\
\left(\nabla\cdot\boldsymbol{J}_{h}^{0},\psi_{h}\right) & =0.
\end{align*}

Similarly, a fully discrete version of the second-order scheme (\ref{eq:SAVAB})
is as follows. For all $n\ge0$, we find $\left(\boldsymbol{u}_{h}^{n+1},p_{h}^{n+1},\boldsymbol{J}_{h}^{n+1},\phi_{h}^{n+1}\right)\in\boldsymbol{X}_{h}\times Y_{h}\times\boldsymbol{D}_{h}\times S_{h}$
and $q_{h}\in\mathbb{R}$ such that for all $\left(\boldsymbol{v}_{h},r_{h},\boldsymbol{K}_{h},\psi_{h}\right)\in\boldsymbol{X}_{h}\times Y_{h}\times\boldsymbol{D}_{h}\times S_{h}$,
\begin{subequations}
\begin{align}
(\delta_{t}^{2}\boldsymbol{u}_{h}^{n+1},\boldsymbol{v}_{h})+R_{e}^{-1}(\nabla\boldsymbol{u}_{h}^{n+1},\nabla\boldsymbol{v}_{h})-(p^{n+1},\nabla\cdot\boldsymbol{v}_{h})\nonumber \\
+q_{h}^{n+1}\exp\left(\frac{t^{n+1}}{T}\right)\left(\left(\hat{\boldsymbol{u}}_{h}^{n+1}\cdot\nabla\hat{\boldsymbol{u}}_{h}^{n+1},\boldsymbol{v}_{h}\right)-\kappa\left(\hat{\boldsymbol{J}}_{h}^{n+1}\times\boldsymbol{B}^{n+1},\boldsymbol{v}_{h}\right)\right) & =0,\label{eq:weakhuAB}\\
\left(\nabla\cdot\boldsymbol{u}_{h}^{n+1},r_{h}\right) & =0,\label{eq:weakhdivuAB}\\
\left(\boldsymbol{J}_{h}^{n+1},\boldsymbol{K}_{h}\right)-\left(\phi_{h}^{n+1},\nabla\cdot\boldsymbol{K}_{h}\right)-q_{h}^{n+1}\exp\left(\frac{t^{n+1}}{T}\right)\left(\hat{\boldsymbol{u}}_{h}^{n+1}\times\boldsymbol{B}^{n+1},\boldsymbol{K}_{h}\right) & =0,\label{eq:weakhJAB}\\
\left(\nabla\cdot\boldsymbol{J}_{h}^{n+1},\psi_{h}\right) & =0,\label{eq:weakhdivJAB}\\
\delta_{t}^{2}q_{h}^{n+1}+\frac{q_{h}^{n+1}}{T}-\exp\left(\frac{t^{n+1}}{T}\right)\nonumber \\
\left(\left(\hat{\boldsymbol{u}}_{h}^{n+1}\cdot\nabla\hat{\boldsymbol{u}}_{h}^{n+1},\boldsymbol{u}_{h}^{n+1}\right)-\kappa\left(\hat{\boldsymbol{u}}_{h}^{n+1}\times\boldsymbol{B}^{n+1},\boldsymbol{J}_{h}^{n+1}\right)-\kappa\left(\hat{\boldsymbol{J}}_{h}^{n+1}\times\boldsymbol{B}^{n+1},\boldsymbol{u}_{h}^{n+1}\right)\right) & =0.\label{eq:weakhqAB}
\end{align}
 \label{eq:SAVABh}
\end{subequations}
The init data is set by the similar way as the first-order scheme
(\ref{eq:SAVEulerh}).

Following the similar procedure as in the proof of Theorems \ref{thm:SeEnergyLaws}
and \ref{thm:SeEnergyLawsAB}, we can obtain the following stability
result.
\begin{thm}
\label{thm:SeEnergyLawsh} The schemes (\ref{eq:SAVEulerh}) and (\ref{eq:SAVABh})
are unconditionally energy stable in the sense that the following energy estimate holds for $n\ge0$,
\begin{align}
\delta_{t}\mathrm{E}_{\text{EL},h}^{n+1} & \le-R_{e}^{-1}\left\Vert \nabla\boldsymbol{u}_{h}^{n+1}\right\Vert ^{2}-\kappa\left\Vert \boldsymbol{J}_{h}^{n+1}\right\Vert ^{2}-\frac{1}{T}\left|q_{h}^{n+1}\right|^{2},\label{eq:SemiEnh}\\
\delta_{t}\mathrm{E}_{\text{BDF},h}^{n+1} & \le-R_{e}^{-1}\left\Vert \nabla\boldsymbol{u}_{h}^{n+1}\right\Vert ^{2}-\kappa\left\Vert \boldsymbol{J}_{h}^{n+1}\right\Vert ^{2}-\frac{1}{T}\left|q_{h}^{n+1}\right|^{2},\label{eq:SemiABh}
\end{align}
where 
\begin{align*}
 & \mathrm{E}_{\text{EL},h}^{n+1}:=\frac{1}{2}\left\Vert \boldsymbol{u}_{h}^{n+1}\right\Vert ^{2}+\frac{1}{2}\left|q_{h}^{n+1}\right|^{2},\\
 & \mathrm{E}_{\text{BDF},h}^{n+1}:=\frac{1}{4}\left(\left\Vert \boldsymbol{u}_{h}^{n+1}\right\Vert ^{2}+\left\Vert 2\boldsymbol{u}_{h}^{n+1}-\boldsymbol{u}_{h}^{n}\right\Vert ^{2}\right)+\frac{1}{4}\left(\left\Vert q_{h}^{n+1}\right\Vert ^{2}+\left\Vert 2q_{h}^{n+1}-q_{h}^{n}\right\Vert ^{2}\right).
\end{align*}
\end{thm}
\begin{proof}
Setting $\left(\boldsymbol{v}_{h},r_{h},\boldsymbol{K}_{h},\psi_{h}\right)=\left(\boldsymbol{u}_{h}^{n+1},p_{h}^{n+1},\kappa\boldsymbol{J}_{h}^{n+1},\phi_{h}^{n+1}\right)$
in (\ref{eq:SAVEulerh}), multiplying (\ref{weakh:q}) by $q^{n+1}$,
and adding the resulting equations yields (\ref{eq:SemiEnh}). The
estimate (\ref{eq:SemiABh}) can be proved in a similar way and we
omit it.
\end{proof}
\medskip{}

By using the non-local and scalar property of the auxiliary variable
$q_{h}^{n+1}$, we can carry out the fully-discrete schemes (\ref{eq:SAVEulerh})
and (\ref{eq:SAVABh}) efficiently as the semi-discrete schemes (\ref{eq:SAVEuler})
and (\ref{eq:SAVAB}) in Section \ref{sec:Schemes}. We leave the
detailed procedures to the interested readers.

\subsection{Accuracy test}

We first verify the first- and second-order accuracy of the proposed
numerical schemes. The computational domain is set as $\Omega=\left(0,1\right)^{d},\:d=2,3$,
and the external magnetic field is $\boldsymbol{B}=\left(0,0,1\right)^{{\rm T}}$.
The physical parameters are given by $R_{e}=\kappa=1$ and the terminal
time $T=1$. The right-hand sides, the initial condition and the Dirichlet
boundary conditions are chosen so that so that the exact solution
is given by
\[
\boldsymbol{u}=\left(y\exp\left(-t\right),x\cos\left(t\right)\right),\quad p=\sin\left(t\right),\quad\boldsymbol{J}=\left(\sin\left(t\right),\cos\left(t\right)\right),\quad\phi=\cos(t),
\]

\noindent for $d=2$ and
\[
\boldsymbol{u}=\left(z\sin\left(t\right),x,y\exp\left(-t\right)\right),\quad p=0,\quad\boldsymbol{J}=\left(\cos\left(t\right),t^{2},0\right),\quad\phi=0.
\]
for $d=3$. 

Note that the exact solutions are linear or
constant in space, the only error comes from the discretization of
the time variable. We fix a mesh size with $h=1/6$ and test the convergence
rate with respect to the time step. The errors and convergence orders
are displayed in Tables \ref{tab:EL2D}-\ref{tab:AB2D} for $d=2$
and Tables \ref{tab:EL3D}-\ref{tab:AB3D} for $d=3$, respectively.
From these tables, we observe that the errors of all variable become
smaller and smaller as the time step is refined. Besides, the corresponding
convergence rates are of the order of $O(\tau)$ for the first-order
scheme and $O(\tau^{2})$ for the second-order scheme asymptotically.
This accord well with our theoretical analysis. 

\begin{table}[h]
\caption{Errors and convergence rates for the first-order scheme (\ref{eq:SAVEuler})
in 2D.\label{tab:EL2D}}

\begin{centering}
\begin{tabular}{|c|c|c|c|c|c|c|}
\hline 
$\tau$ & $\left\Vert e_{\boldsymbol{u}}^{N}\right\Vert $ & $\left\Vert \nabla e_{\boldsymbol{u}}^{N}\right\Vert $ & $\|e_{p}^{N}\|$ & $\|e_{\boldsymbol{J}}^{N}\|_{\text{div }}$ & $\|e_{\phi}^{N}\|$ & $\left|e_{q}^{N}\right|$\\
\hline 
0.2 & 2.13e-04(\textemdash ) & 1.74e-03(\textemdash ) & 3.53e-02(\textemdash ) & 7.88e-06(\textemdash ) & 2.19e-02(\textemdash ) & 1.27e-02(\textemdash )\\
\hline 
0.1 & 1.03e-04(1.06) & 8.35e-04(1.06) & 1.87e-02(0.92) & 3.51e-06(1.17) & 1.19e-02(0.88) & 4.88e-03(1.38)\\
\hline 
0.05 & 5.02e-05(1.03) & 4.09e-04(1.03) & 9.65e-03(0.96) & 1.64e-06(1.10) & 6.16e-03(0.95) & 2.01e-03(1.28)\\
\hline 
0.025 & 2.49e-05(1.02) & 2.02e-04(1.02) & 4.90e-03(0.98) & 7.88e-07(1.06) & 3.14e-03(0.97) & 8.93e-04(1.17)\\
\hline 
0.0125 & 1.24e-05(1.01) & 1.01e-04(1.01) & 2.47e-03(0.99) & 3.86e-07(1.03) & 1.58e-03(0.99) & 4.17e-04(1.10)\\
\hline 
\end{tabular}
\par\end{centering}
\caption{Errors and convergence rates for the second-order scheme (\ref{eq:SAVAB})
in 2D.\label{tab:AB2D}}

\begin{centering}
\begin{tabular}{|c|c|c|c|c|c|c|}
\hline 
$\tau$ & $\left\Vert e_{\boldsymbol{u}}^{N}\right\Vert $ & $\left\Vert \nabla e_{\boldsymbol{u}}^{N}\right\Vert $ & $\|e_{p}^{N}\|$ & $\|e_{\boldsymbol{J}}^{N}\|_{\text{div }}$ & $\|e_{\phi}^{N}\|$ & $\left|e_{q}^{N}\right|$\\
\hline 
0.2 & 3.33e-05(\textemdash ) & 2.71e-04(\textemdash ) & 7.23e-03(\textemdash ) & 9.06e-07(\textemdash ) & 4.54e-03(\textemdash ) & 3.51e-03(\textemdash )\\
\hline 
0.1 & 8.42e-06(1.98) & 6.86e-05(1.98) & 1.68e-03(2.10) & 2.55e-07(1.83) & 1.02e-03(2.16) & 1.01e-03(1.80)\\
\hline 
0.05 & 2.12e-06(1.99) & 1.73e-05(1.99) & 4.31e-04(1.97) & 6.42e-08(1.99) & 2.40e-04(2.08) & 2.35e-04(2.11)\\
\hline 
0.025 & 5.32e-07(2.00) & 4.33e-06(2.00) & 1.11e-04(1.96) & 1.61e-08(2.00) & 5.82e-05(2.04) & 5.42e-05(2.11)\\
\hline 
0.0125 & 1.33e-07(2.00) & 1.08e-06(2.00) & 2.81e-05(1.98) & 4.03e-09(2.00) & 1.43e-05(2.02) & 1.29e-05(2.07)\\
\hline 
\end{tabular}
\par\end{centering}
\caption{Errors and convergence rates for the first-order scheme (\ref{eq:SAVEuler})
in 3D\label{tab:EL3D}}

\begin{centering}
\begin{tabular}{|c|c|c|c|c|c|c|}
\hline 
$\tau$ & $\left\Vert e_{\boldsymbol{u}}^{N}\right\Vert $ & $\left\Vert \nabla e_{\boldsymbol{u}}^{N}\right\Vert $ & $\|e_{p}^{N}\|$ & $\|e_{\boldsymbol{J}}^{N}\|_{\text{div }}$ & $\|e_{\phi}^{N}\|$ & $\left|e_{q}^{N}\right|$\\
\hline 
0.2 & 4.21e-04(\textemdash ) & 3.90e-03(\textemdash ) & 1.56e-01(\textemdash ) & 2.43e-02(\textemdash ) & 5.66e-02(\textemdash ) & 1.32e-01(\textemdash )\\
\hline 
0.1 & 1.99e-04(1.08) & 1.85e-03(1.08) & 7.67e-02(1.02) & 1.62e-02(0.58) & 3.08e-02(0.88) & 6.90e-02(0.93)\\
\hline 
0.05 & 9.93e-05(1.01) & 9.26e-04(1.00) & 3.78e-02(1.02) & 9.06e-03(0.84) & 1.60e-02(0.95) & 3.51e-02(0.97)\\
\hline 
0.025 & 4.99e-05(0.99) & 4.67e-04(0.99) & 1.88e-02(1.01) & 4.76e-03(0.93) & 8.13e-03(0.97) & 1.77e-02(0.99)\\
\hline 
0.0125 & 2.51e-05(0.99) & 2.35e-04(0.99 ) & 9.34e-03(1.01 ) & 2.44e-03(0.97) & 4.10e-03(0.99) & 8.89e-03(0.99)\\
\hline 
\end{tabular}
\par\end{centering}
\caption{Errors and convergence rates for the second-order scheme (\ref{eq:SAVAB})
in 3D.\label{tab:AB3D}}

\centering{}%
\begin{tabular}{|c|c|c|c|c|c|c|}
\hline 
$\tau$ & $\left\Vert e_{\boldsymbol{u}}^{N}\right\Vert $ & $\left\Vert \nabla e_{\boldsymbol{u}}^{N}\right\Vert $ & $\|e_{p}^{N}\|$ & $\|e_{\boldsymbol{J}}^{N}\|_{\text{div }}$ & $\|e_{\phi}^{N}\|$ & $\left|e_{q}^{N}\right|$\\
\hline 
0.2 & 2.38e-04(\textemdash ) & 2.24e-03(\textemdash ) & 6.03e-02(\textemdash ) & 3.04e-02(\textemdash ) & 3.43e-02(\textemdash ) & 5.84e-02(\textemdash )\\
\hline 
0.1 & 6.26e-05(1.93) & 5.89e-04(1.93) & 1.54e-02(1.97) & 7.86e-03(1.95) & 8.93e-03(1.94) & 1.53e-02(1.93)\\
\hline 
0.05 & 1.60e-05(1.97) & 1.50e-04(1.97) & 3.90e-03(1.98) & 2.00e-03(1.97) & 2.28e-03(1.97) & 3.90e-03(1.97)\\
\hline 
0.025 & 4.03e-06(1.98) & 3.80e-05(1.98) & 9.81e-04(1.99) & 5.05e-04(1.99) & 5.74e-04(1.99) & 9.85e-04(1.99)\\
\hline 
0.0125 & 1.01e-06(1.99) & 9.54e-06(1.99) & 2.46e-04(2.00) & 1.27e-04(1.99) & 1.44e-04(1.99) & 2.47e-04(1.99)\\
\hline 
\end{tabular}
\end{table}

\subsection{Stability test}

This example is devoted to test the stability of the of the proposed
SAV schemes. We set the computed domain to be $\Omega=\left(0,1\right)^{d},\:d=2,3$
and the external magnetic field to be $\boldsymbol{B}=\left(0,0,1\right)^{{\rm T}}$.
The physical parameters are given by $R_{e}=\kappa=20,100$ and the
terminal time $T=3$. The initial condition is chosen as 
\[
\begin{array}{l}
\boldsymbol{u}^{0}=\left(\sin\left(\pi x\right)\cos\left(\pi y\right),-\cos\left(\pi x\right)\sin\left(\pi y\right)\right).\end{array}
\]
for $d=2$ and
\begin{align*}
\boldsymbol{u}^{0} & =-\pi/2\sin\left(\pi x\right)\sin\left(\pi y\right)\sin\left(\pi z\right)\boldsymbol{\varPsi},\\
\boldsymbol{\varPsi} & =\left(\sin\left(\pi x\right)\cos\left(\pi y\right)\cos\left(\pi z\right),-2\cos\left(\pi x\right)\sin\left(\pi y\right)\cos\left(\pi z\right),\cos\left(\pi x\right)\cos\left(\pi y\right)\sin\left(\pi z\right)\right)^{{\rm T}}.
\end{align*}
for $d=3$. 

With the prescribed data, we test the energy stability of the SAV
schemes on the fixed mesh size with $h=1/150$ for 2D and $h=1/16$
for 3D. Figures. 1-2 present the time evolution of the energy for
different time steps. We observe that all energy curves decay monotonically
for all time step sizes in both 2D and 3D. This confirms that the
SAV schemes are unconditionally energy stable.

\begin{figure}
\hfill{}\subfloat[First-order SAV scheme with $R_{e}=\kappa=20$.]{\begin{centering}
\includegraphics[scale=0.6]{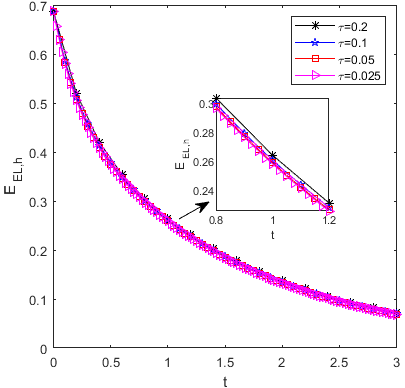}
\par\end{centering}
}\hfill{}\subfloat[First-order SAV scheme with $R_{e}=\kappa=100$.]{\begin{centering}
\includegraphics[scale=0.6]{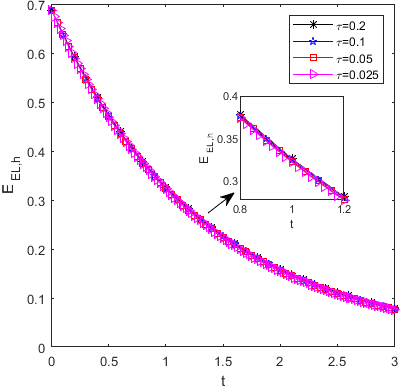}
\par\end{centering}
}\hfill{}

\hfill{}\subfloat[Second-order SAV scheme with $R_{e}=\kappa=20$.]{\begin{centering}
\includegraphics[scale=0.6]{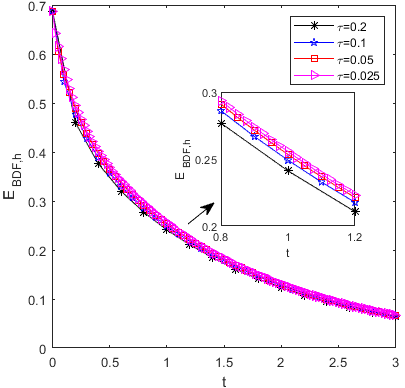}
\par\end{centering}
}\hfill{}\subfloat[Second-order SAV scheme with $R_{e}=\kappa=100$.]{\begin{centering}
\includegraphics[scale=0.6]{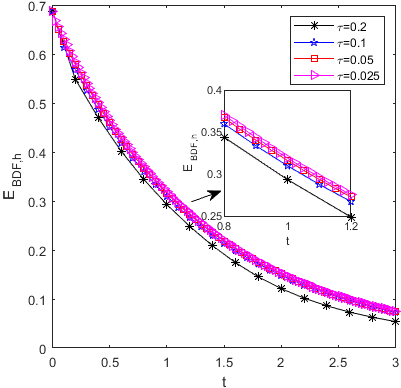}
\par\end{centering}
}\hfill{}

\caption{Time evolution of the energy for different time step sizes in 2D.\label{fig:Energy2D}}
\end{figure}

\begin{figure}
\hfill{}\subfloat[First-order SAV scheme with $R_{e}=\kappa=20$.]{\begin{centering}
\includegraphics[scale=0.6]{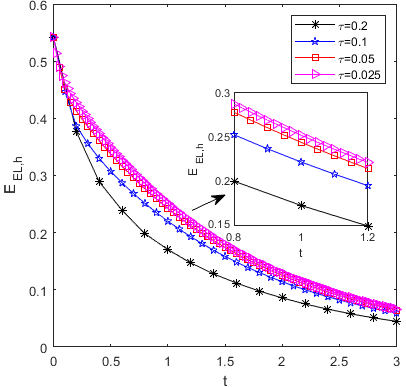}
\par\end{centering}
}\hfill{}\subfloat[First-order SAV scheme with $R_{e}=\kappa=100$.]{\begin{centering}
\includegraphics[scale=0.6]{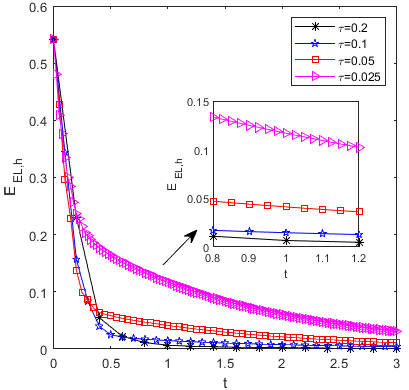}
\par\end{centering}
}\hfill{}

\hfill{}\subfloat[Second-order SAV scheme with $R_{e}=\kappa=20$.]{\begin{centering}
\includegraphics[scale=0.6]{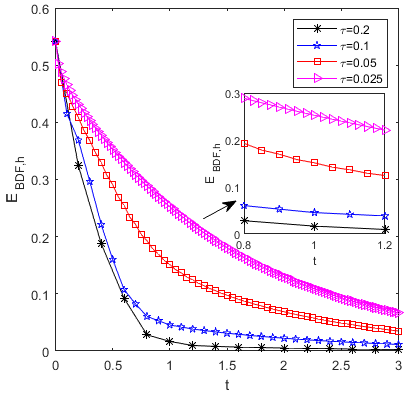}
\par\end{centering}
}\hfill{}\subfloat[Second-order SAV scheme with $R_{e}=\kappa=100$.]{\begin{centering}
\includegraphics[scale=0.6]{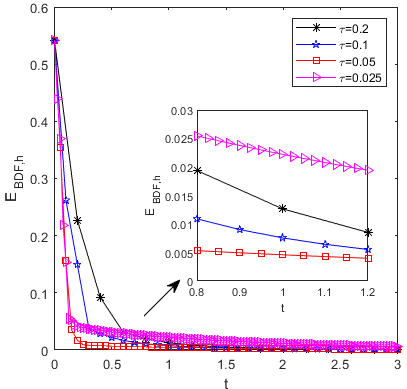}
\par\end{centering}
}\hfill{}

\caption{Time evolution of the energy for different time step sizes in 3D.\label{fig:Energy3D}}
\end{figure}

\subsection{Lid driven cavity}

In this example, we consider a well-known benchmark problem in fluid
dynamics, known as lid-cavity flow. For this end, we assume that the
cavity is a unit cubic in 3D. The physical parameters are set by $R_{e}=200,\kappa=10$,
The applied magnetic field is $\boldsymbol{B}=\left(1,0,0\right)^{{\rm T}}$and
the initial values are given by $\boldsymbol{u}_{0}=\left(g_{1},0,0\right)^{{\rm T}}$,
where $g_{1}=g_{1}(z)$ is a continuous function and satisfies
\[
g_{1}(x,y,1)=1,\quad\text{ and }\quad g_{1}(x,y,z)=0\quad\forall z\in[0,1-h],
\]
where $h$ is the mesh size. The boundary conditions are set by
\[
\boldsymbol{u}=\boldsymbol{u}_{0},\quad\phi=0\quad\text{on}\quad\Gamma.
\]
Note that our schemes apply equally to the above boundary conditions.

For this problem, we want to see how the fluid flows under the influence of the magnetic field. We perform the numerical tests by using
the second-order SAV scheme with the mesh size $h=1/32$ and
the time step $\tau=0.01$. For convenience, the terminal time $T$
is set by $T=10$. Figure \ref{fig:Lid3D} displays the streamlines
of $\boldsymbol{u}_{h}$ and the the distributions of \textbf{$\left|\boldsymbol{J}\right|_{h}$}
on the cross-section $y=0.5$ at the terminal time. It can be seen
that the structure of vortex is similar to those reported in \citep{Zhang2021,Li191}
where the steady inductionless MHD equations are considered. To investigate the formation of the final vortex, we show some snapshots of
the streamlines of $\boldsymbol{u}_{h}$ on the cross-section $y=0.5$
in Figure \ref{fig:Lid3Dt}. Our numerical results indicate that the
fluid yields more large vertices and tends to be stratified as
time evolves until the physical fields reach steady states. 

\begin{figure}
\begin{centering}
\begin{tabular}{cc}
\includegraphics[scale=0.45]{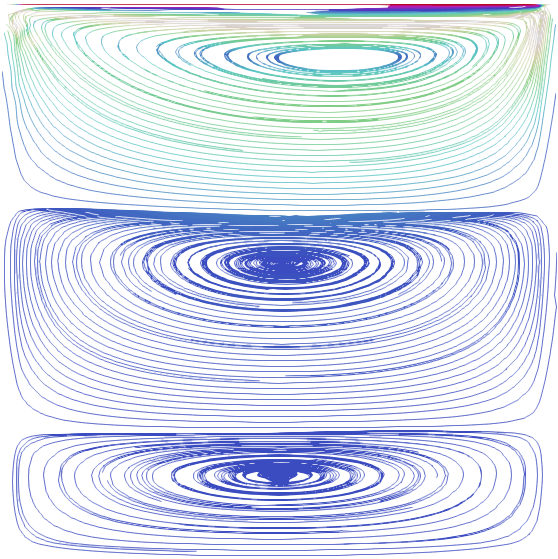} & \includegraphics[scale=0.45]{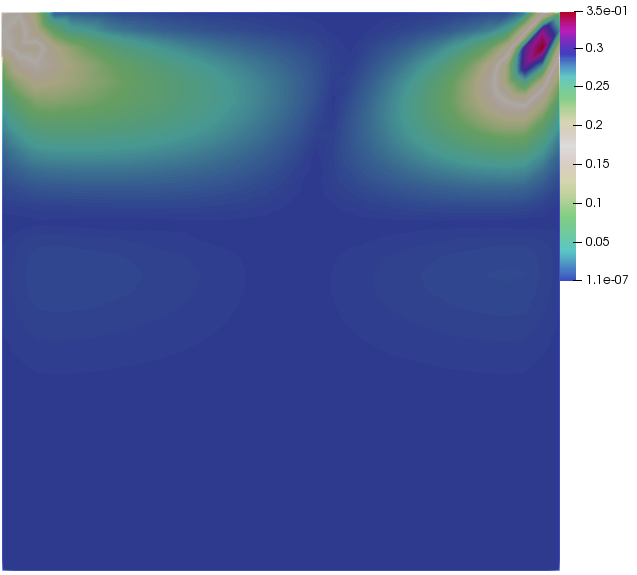}\\
\end{tabular}
\par\end{centering}
\caption{Streamlines of \textbf{$\boldsymbol{u}_{h}$} and distributions of\textbf{$\left|\boldsymbol{J}\right|_{h}$}
on the cross-section $y=0.5$ (Right).\label{fig:Lid3D}}
\end{figure}

\begin{figure}
\begin{centering}
\begin{tabular}{ccc}
\includegraphics[scale=0.36]{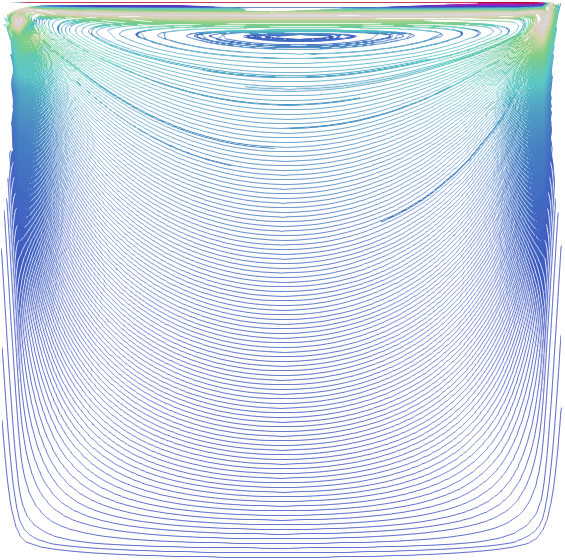} & \includegraphics[scale=0.35]{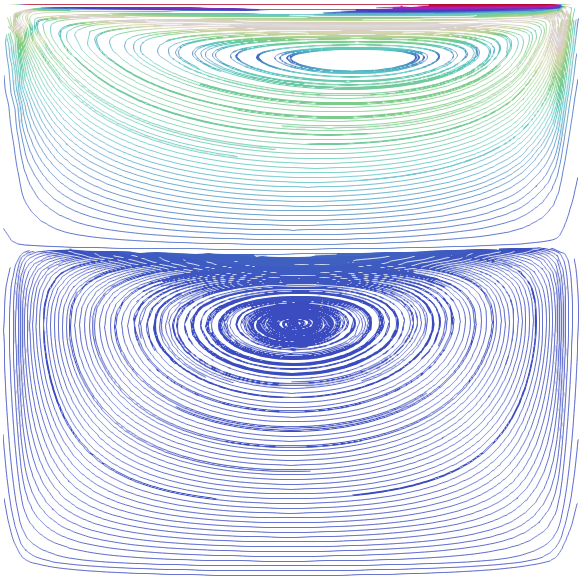} & \includegraphics[scale=0.42]{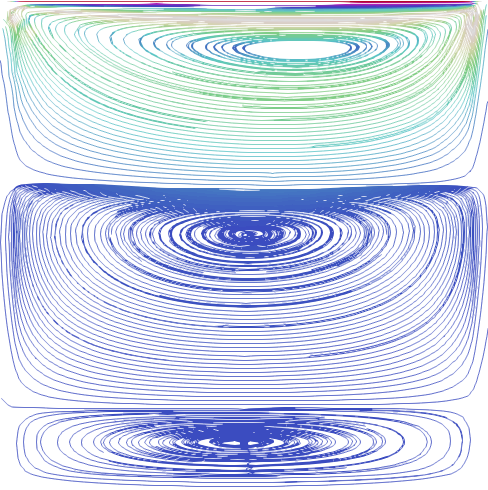}\\
\end{tabular}
\par\end{centering}
\caption{Time-evolution of streamlines of \textbf{$\boldsymbol{u}_{h}$} on
the cross-section $y=0.5$. Left: $t=0.1$, Middle: $t=1$, Right:
$t=2$.\label{fig:Lid3Dt}}
\end{figure}

\section{Concluding remarks\label{sec:Conclud}}

In this paper, we propose and analyze some SAV schemes for inductionless
MHD equations. The attractive points of these schemes are they are
decoupled, linear, unconditionally energy stable and easy to implement.
We further derive rigorous error estimates for the first-order scheme
in the two-dimensional case without any condition on the time step.
A series of numerical experiments are given to confirm the theoretical
findings and show the performances of the schemes. 

Remarkably, we only present the error analysis for the first-order
scheme in the two-dimensional case. We believe that the error estimates
can also be established for the second-order scheme in the two-dimensional
case, although the process will surely be much more tedious. However,
it appears that the error estimates can not be easily extended to
the three-dimensional case, as our proof uses essentially some inequalities
which are only valid in the two-dimensional case. In the further,
the error estimates in three dimensions will be considered. Moreover,
we have only considered time discretization in this work. Error analysis
for full discretization will be left as a subject of future endeavors.

\section*{References}

\bibliographystyle{unsrt}
\bibliography{ref}

\end{document}